\documentclass[12pt, reqno]{amsart}

\usepackage[margin=0.75in]{geometry}
\usepackage[
  bookmarks=true]{hyperref}
\usepackage[OT2, T1]{fontenc}
\usepackage[english]{babel}
\usepackage[utf8]{inputenc}
\usepackage{csquotes}
\usepackage[final]{microtype}
\usepackage{lmodern}
\usepackage{amsthm}
\usepackage{amssymb}
\usepackage{mathrsfs}
\usepackage{enumerate}
\usepackage{tikz-cd} 
\usepackage{tikz}
\usepackage{multicol}
\usepackage{multirow}
\usepackage{tikz-qtree}
\usepackage{rotating}
\usepackage{graphicx}
\usepackage{comment}
\usepackage{enumitem}
\usepackage{manfnt}
\usetikzlibrary{arrows,calc,matrix,trees,arrows.meta,positioning,decorations.pathreplacing,bending}

\newtheorem{theorem}{Theorem}[section]
\newtheorem*{theorem*}{Theorem}
\newtheorem*{goal*}{Goal}
\newtheorem*{question*}{Question}

\newtheorem{proposition}[theorem]{Proposition}
\newtheorem{lemma}[theorem]{Lemma}
\newtheorem{corollary}[theorem]{Corollary}

\theoremstyle{definition}

\newtheorem{remark}[theorem]{Remark}
\newtheorem{definition}[theorem]{Definition}
\newtheorem{example}[theorem]{Example}

\newcommand{\Sel}{\mathrm{Sel}} 

\newcommand{\Hom}{\mathrm{Hom}}

\newcommand{\Res}{\mathrm{Res}}

\newcommand{\genlegendre}[4]{%
  \genfrac{(}{)}{}{#1}{#3}{#4}%
  \if\relax\detokenize{#2}\relax\else_{\!#2}\fi
}

\DeclareSymbolFont{cyrletters}{OT2}{wncyr}{m}{n}
\DeclareMathSymbol{\Sha}{\mathalpha}{cyrletters}{"58}

\begin{document}

\title[]{Rank growth of elliptic curves over $S_3$ extensions with fixed quadratic resolvents}
\author{Daniel Keliher}
\address{Daniel Keliher, Massachusetts Institute of Technology, Concourse Program, Cambridge, MA 02139, USA}
\email{keliher@mit.edu}
\urladdr{\url{https://www.danielkeliher.com/}}

\author{Sun Woo Park}
\address{Sun Woo Park, Max Planck Institute for Mathematics,  Vivatsgasse 7, 53111 Bonn, German}
\email{s.park@mpim-bonn.mpg.de}
\urladdr{\url{https://sites.google.com/wisc.edu/spark483}}

\begin{abstract}
We study the probability with which an elliptic curve $E/k$, subject to some technical conditions, gains rank upon base extension to an $S_3$-cubic extension $K/k$ with quadratic resolvent field $F/k$,  all three fields of which are subject to some mild technical conditions. To do so, we determine the distribution (under a non-standard ordering) of Selmer ranks of an auxiliary abelian variety associated to $E$ and $S_3$-cubic extensions $K/k$ following ideas of Klagsbrun, Mazur, and Rubin. One corollary of this distribution is that $E$ gains rank by at most one upon base extension to $K$ with probability at least $31.95\%$.
\end{abstract}
\maketitle
\tableofcontents

\section{Introduction}

Questions related to expected ranks (or ranks of Selmer groups) of elliptic curves have a rich history. In \cite{SDtwists}, Swinnerton-Dyer determined the probability distribution of the rank of the 2-Selmer group of quadratics twists by $d$ of elliptic curve with full rational two torsion as the number of prime factors of $d$ grows arbitrarily large. This was the first work to view the successive manipulation of local conditions of the Selmer group as a Markov process. In \cite{Kane2013}, Kane obtained the same distribution in the same setting, but without the ``non-standard'' ordering on twists adopted by Swinnerton-Dyer.

In \cite{KMR14}, Klagsbrun, Mazur, and Rubin determined the average size of the $p$-Selmer ranks of an elliptic curve meeting some mild conditions. Their work similarly adopts this Markov chain perspective and relies on a different non-standard ordering, which we likewise adopt in this paper  (see Definition \ref{def:FAN}). Explicit versions of the results of \cite{KMR14} were proved in the function field setting by the second author \cite{park2022prime}.  

Finally, a series of papers by Smith \cite{Smith1}, \cite{Smith2}, determined the distribution of $2^\infty$-Selmer coranks, essentially resolving (assuming the finiteness of the Tate-Shafarevich group) a conjecture of Goldfeld \cite{GoldfeldConjecture}, under the usual ordering of quadratic twists.

Results like these are also connected to enumerating field extensions (say, with prescribed degree and Galois group) for which a fixed elliptic curve either does or does not gain rank upon base extension.  For quadratic extensions, this is Goldfeld's conjecture, but a bit more is known. Lemke Oliver and Thorne \cite{LOT} gave lower bounds for the number of degree $n \geq 2$, $S_n$-extensions $K/\mathbb{Q}$ for which a fixed $E/\mathbb{Q}$ \textit{does} gain rank upon base extension to $K$. The first author \cite{keliherS4} gave lower bounds for the number of quartic $A_4$ and $S_4$-extensions for which a fixed $E/\mathbb{Q}$ of rank zero \textit{does not} gain rank upon base change. Pathak and Ray \cite{PathakRay} likewise gave lower bounds on the number of meta-abelian extensions for which a fixed curve of rank zero does not gain rank. For all three results, the lower bounds on the number of extensions meeting the prescribed rank growth behavior are a density zero (but still an infinitely large) set among the set of degree $n$ number fields with bounded discriminant and specified Galois group when compared to known asymptotics, or to Malle's predictions \cite{Malle1}, \cite{Malle2}, for the number of such extensions. 

In this present work, our motivating question is the following. For a fixed elliptic curve $E$ defined over a number field $k$, how often does the rank of $E$ change upon base extension to an $S_3$-cubic extension of $k$? Here we deal with the generic case in which the Galois group of the 3-torsion field of $E/k$ contains $\mathrm{SL}_2(\mathbb{F}_3)$.

A different manifestation of this problem comes via studying Selmer groups of cubic twists family of an elliptic curve with $j$-invariant zero. Producing a trivial Selmer group for such a twist of a curve by $d$ essentially produces a cubic field obtained by adjoining $\sqrt[3]{d}$ for which the original curve does not gain rank upon base change.   In \cite{ABS}, Alp{\"o}ge, Bhargava and Shnidman determined the average size of 2-Selmer groups in such a family. The average size of the 3-Selmer groups of the same family was determined by Koymans and Smith in \cite{KoymansSmith}.

\subsection{Main Results} 
Let $k$ be a number field and let $F/k$ be a quadratic extension. We classify $S_3$-cubic extensions of $k$ with \emph{quadratic resolvent} $F$, i.e. the quadratic extension intermediate to the Galois closure of $K/k$ and $k$, into three categories 
\begin{center}
    (A) When $\zeta_3 \in k \subset F$; \quad 
    (B) When $\zeta_3 \notin k$  and $F = k(\zeta_3)$; \quad 
    (C) When $\zeta_3 \notin F$.
\end{center}
These cases emerge naturally from Lemma \ref{lem:S3byQuad} and are explained further in Remark \ref{ex:S3byA}. We fully describe the allowable quadratic extensions $F/k$ for our results in Definition \ref{def:admissible}. 

We first state a coarse version of our results which gives a lower bound on how often one might expect the rank of $E/k$ to grow by at most one upon base extension to an $S_3$-cubic extension. 

\begin{theorem} \label{theorem:mainA} Fix a number field $k$ and a quadratic extension $F/k$ which is admissible in the sense of Definition \ref{def:admissible}. Fix an elliptic curve $E/k$ for which $\mathrm{Gal}(k(E[3])/k) \supset \mathrm{SL}_2(\mathbb{F}_3)$,
and $F$ is linearly disjoint from $k(E[3])$ over $k$. 
Then the density, with respect to the fan structure (see Definition \ref{def:FAN}), of $S_3$-cubic extensions  $K/k$ with quadratic resolvent $F/k$ for which $\mathrm{rk}(E/K) \leq \mathrm{rk}(E/k) + 1$ is at least $31.95\%$.
\end{theorem}

The ordering on the $S_3$-cubic extensions (from which we obtain various densities) is according to the \textit{fan structure}. For further discussion of this ordering, see Remark \ref{rmk:fan}.  

Given an $S_3$-cubic extension $K/k$, there exists an auxiliary four dimensional abelian variety $B_{K/k}$ depending on $K/k$, its quadratic resolvent $F/k$, and $E/k$ (see Definition \ref{def:4-dimAV} and the surrounding discussion) whose rank of $k$-rational points satisfies the relation 
$$\mathrm{rk}(B_{K/k}(k))= 2\left(\mathrm{rk}(E(K)) - \mathrm{rk}(E(k))\right).$$
That is, the variety $B_{K/k}$ measures the rank growth of $E$ in $S_3$-cubic extensions. 

Theorem \ref{theorem:mainA} follows from the following result on the probability distribution of dimensions of $1 - \sigma_K$ Selmer groups of families of abelian varieties $B_{K/k}$ (denoted as $\Sel_{1-\sigma_K}(B_{K/k}/k)$), where $\sigma_K \in \text{Gal}(\tilde{K}/k)$ is any element of order $3$. Their dimensions as $\mathbb{F}_3$-vector spaces give an upper bound on rank growths of $E$ with respect to $K/k$:
\begin{equation*}
    \dim_{\mathbb{F}_3} \Sel_{1-\sigma_K}(B_{K/k}/k) \geq \mathrm{rk}(E(K)) - \mathrm{rk}(E(k)).
\end{equation*}

\begin{theorem}\label{thm:mainB} Fix a number field $k$ and a quadratic extension $F/k$ which is admissible in the sense of Definition \ref{def:admissible}. Fix an elliptic curve $E/k$ for which $\mathrm{Gal}(k(E[3])/k) \supset \mathrm{SL}_2(\mathbb{F}_3)$
and $F$ is linearly disjoint from $k(E[3])$ over $k$.
Then the density, with respect to the fan structure (see Definition \ref{def:FAN}), of $S_3$-cubic extensions  $K/k$ with quadratic resolvent $F$ for which the variety $B_{K/k}$ satisfies $\dim_{\mathbb{F}_3}\Sel_{1-\sigma_K}(B_{K/k}/k) = s$
is
\begin{equation*}
    \begin{cases}
        \rho_E  \cdot \prod_{k=0}^\infty \frac{1}{1 + 3^{-k}} \cdot \prod_{k=1}^{\frac{s}{2}} \frac{3}{3^k - 1} &\text{ if } s \text{ is even}, \\
        (1-\rho_E) \cdot \prod_{k=0}^\infty \frac{1}{1 + 3^{-k}} \cdot \prod_{k=1}^{\frac{s-1}{2}} \frac{3}{3^k - 1} &\text{ if } s \text{ is odd},        
    \end{cases}
\end{equation*}
where $\rho_E \in [0,1]$ is a computable constant depending only on $E$ and is given by Definition \ref{def:rhoE}.
\end{theorem}
We list the percentage values of the probability distribution of the dimensions of Selmer groups for the case where $\rho_E = 1$ and $\rho_E = 0$. Each respective case corresponds to families of $S_3$-cubics whose dimensions of Selmer groups of $B_{K/k}$ are all even or all odd. The respective probability distribution for different values of $\rho_E \in (0,1)$ can be obtained from linearly interpolating the two probabilities in each row. In particular, the density, with respect to the fan structure, of $S_3$-cubic extensions $K/k$ wtih quadratic resolvent $F/k$ for which $\text{rk}(E/K) = \text{rk}(E/k)$ is at least $\rho_E \cdot 31.95\%$.
\begin{center}
    \begin{tabular}{c||cc}
         & \multicolumn{2}{c}{Density}  \\
        $s$ & $\rho_E = 1$ & $\rho_E = 0$ \\
        \hline \hline 
        0 & 31.9502 & 0 \\
        1 & 0 & 31.9502 \\
        2 &47.9253 & 0\\
        3 & 0& 47.9253 \\
        4 & 17.9720 & 0\\
        5 & 0 & 17.9720 \\
        6 &2.07369 & 0\\
        7 & 0 & 2.07369  \\
        8 & $7.77635 \times 10^{-2}$ & 0 \\
        9 & 0 & $7.77635 \times 10^{-2}$\\
        \hline
        $\geq 10$ & $\leq 9.67988 \times 10^{-4}$ &  $\leq 9.67988 \times 10^{-4}$\\
    \end{tabular}
\end{center}
This theorem is an immediate consequence of Theorem \ref{thm:main-full}. From Corollary \ref{cor:exponential-decay}, we see that the proportion of extensions exhibiting larger rank growth decays quadratic-exponentially.

\subsection{Layout} In Section \ref{sec:Selmer}, we introduce the necessary machinery for discussing Selmer groups of abelian varieties. In Section \ref{section:S3-param}, we explain a very explicit parameterization of $S_3$-cubic extensions with a fixed quadratic resolvent field following \cite{CohenMorra}. In Section \ref{sec:measure-rank-growth}, we show how the measuring rank growth in $S_3$-cubic extensions can be obtained from understanding rank growths in an associated $C_3$-cyclic extensions and constructing four dimensional abelian varieties whose Selmer groups encode much of this data. In Sections \ref{sec:markov-chains-ranks} and \ref{sec:fans}, we explain how to model rank growth questions as one of a distribution of Selmer ranks. Much of the material in these sections is adapted or expanded from \cite{KMR13, KMR14} and tailored to our setting. Finally, in Section \ref{sec:proof-of-main-thms}, we prove our main result, Theorem \ref{thm:main-full}, from which Theorems \ref{theorem:mainA} and \ref{thm:mainB} follow directly.

\subsection{Acknowledgements} This project began after a JMM 2023 special session on arithmetic statistics; we thank the organizers for inviting us. Much of the preliminary work was completed at a PANTS meeting at the University of Georgia, we thank the grant organizers for making such a meeting possible. We are indebted to Jiuya Wang and Robert Lemke Oliver for their guidance and many helpful conversations. We thank Jordan Ellenberg for his input on the proof of Proposition \ref{prop:torsion}. Finally, we thank Valentin Blomer, Ashay Burungale, Michael Daas, Zev Klagsbrun, Peter Koymans, Alex Smith, and Douglas Ulmer. Most of this work was completed while the first author was a postdoc at the University of Georgia and while the second was a graduate student at the University of Wisconsin, Madison and a postdoc at MPIM; we thank these institutions for their support.

\section{Selmer Groups}\label{sec:Selmer}
Here, we briefly introduce Selmer groups of abelian varieties. For a more thorough treatment, see e.g. \cite{Silverman1} or \cite{Cassels}.

Throughout this section, let $A$ and $A'$ be abelian varieties over a number field, $k$, together with an isogeny $\phi: A \to A'$. From these, one has the short exact sequence 
$$0 \to \ker \phi \to A \overset{\phi}{\to} A' \to 0.$$
After taking the long exact sequence on Galois cohomology groups and quotienting it appropriately, one arrives at the short exact sequence 
$$0 \to A'(k)/\phi(A(k)) \to H^1(k, \ker \phi) \to H^1(k, A)[\phi] \to 0.$$
Above, the map $A'(k)/\phi(A(k)) \to H^1(k, \ker \phi)$, the global Kummer map, is given by sending  $\bar P \in A'(k)/\phi(A(k))$ to the class in $H^1(k, \ker \phi)$ of maps all cohomologous to the map $G_k \to \ker \phi$ given by $\sigma \mapsto \bar P^\sigma - \bar P$. One likewise defines the local Kummer maps with $k_v$ in place of $k$ for every place $v$ of $k$. 

Further, for each place $v$ of $k$, there are restrictions maps:
$$A'(k)/\phi(A(k)) \to A'(k_v)/\phi(A(k_v))$$
given by inclusion of points; there are also  
$$H^1(k, \ker \phi) \to H^1(k_v, \ker \phi) \quad \text{and} \quad H^1(k, A) \to H^1(k_v, A)$$
given by the restriction of maps for the absolute Galois group $G_k$ to the decomposition group at $v$. Finally, define the \textit{restricted cohomology} group by 
$$H^1_f(k_v, \ker \phi) := \mathrm{Image}\left( A'(k_v)/\phi(A(k_v)) \to  H^1(k_v, \ker \phi) \right).$$

\begin{definition}\label{def:selmergroup}
    The \emph{$\phi$-Selmer group}, denoted $\mathrm{Sel}_\phi(A/k)$, is given by the exactness of 
    $$0 \to A'(k)/\phi(A(k)) \to \mathrm{Sel}_\phi(A/k) \to \prod_v H^1(k_v, \ker \phi)/H^1_f(k_v, \ker \phi).$$
\end{definition}

\begin{example}[The $p$-Selmer group of an elliptic curve]
    Fix a prime $p$ and let $E$ be an elliptic curve over $k$. Consider the multiplication-by-$p$ maps $[p]:E \to E$. Then 
    $$0 \to E(k)/pA(k) \to \mathrm{Sel}_p(E/k) \to \prod_v H^1(k_v, E[p])/H^1_f(k_v, E[p])$$
    gives the $p$-Selmer group of $E$.
\end{example}

\begin{remark}
    For an abelian variety $A$ defined over a number field $k$, and $p$ a prime, we have that 
    $$\mathrm{rk}(A(k)) \leq \dim_{\mathbb{F}_p}\mathrm{Sel}_p(A/k).$$
    A major theme will be to control $\mathrm{Sel}_p(A/k)$, namely force it to have dimension zero, by judicious choices of twists. This in turn gives us an understanding of the rank of $A(k)$. For instance, if $E^d/k$ is the quadratic twist of an elliptic curve $E/k$ by some $d \in k^\times/(k^\times)^2$, then with $F=k(\sqrt{d})$, one has
    $$\mathrm{rk}(E^d(k)) = \mathrm{rk}(E(F)) - \mathrm{rk}(E(k)) \leq \dim_{\mathbb{F}_2} \mathrm{Sel}_2(E^d/k).$$ 
    The idea is then to understand the 2-Selmer group of $E^d$ from the data of $\mathrm{Sel}_2(E/k)$ and $F=k(\sqrt{d})$. See, for example, \cite{MR10}, where, among other things, the authors bound the number of $F$ for which the $\mathrm{Sel}_2(E^F/k)$ has a prescribed $\mathbb{F}_2$-dimension under some mild conditions.
\end{remark}

\section{Parameterizing \texorpdfstring{$S_3$}{S3} Cubic Extensions} \label{section:S3-param}
In this section we recall, for a quadratic extension of number fields $F/k$, how to construct $S_3$-cubic extensions $K/k$ for which the unique quadratic subfield of the Galois closure $\tilde K/k$ of $K/k$ is $F/k$. That is, $\tilde K/k$ is a Galois $S_3$-extension containing the $S_3$-cubic extension $K/k$ and has unique quadratic sub-extension $F/k$. We follow the set-up of \cite{CohenMorra} and modify their notations somewhat to suit our purposes. 

Let $k$ be a number field, and let $\zeta_3$ be a primitive cube root of unity. Throughout, $F/k$ will denote a quadratic extension of $k$. Now, let $F_z = F(\zeta_3)$ and $k_z = k(\zeta_3)$. Let $\tau_2$ be the generator of $\mathrm{Gal}(F/k) \simeq C_2$ and $\tau$ be the generator of $\mathrm{Gal}(F_z/F)$. 

$S_3$-cubic extensions $K/k$ with quadratic resolvent $F$ depend on $\tau, \tau_2$ and their actions on $\zeta_3$. We need three cases, for each of which we define a subset $T$ of the group ring $\mathbb{Z}[\mathrm{Gal}(F_z/k)]$.

\begin{itemize}
    \item[(i)] $T=\{\tau_2 + 1\}$ when $\tau$ is trivial and $\tau_2$ fixes $\zeta_3$; in this case $\zeta_3 \in k$.
    \item[(ii)] $T=\{\tau_2-1\}$ when $\tau$ is trivial and $\tau_2$ does not fix $\zeta_3$; in this case $F=k(\zeta_3)$.
    \item[(iii)] $T=\{\tau + 1, \tau_2 +1\}$ when $\tau$ is nontrivial and does not fix $\zeta_3$; in this case $\zeta_3 \notin F$.
\end{itemize}
Note that cases (i) , (ii), and (iii) above correspond to cases (3), (4), and (5) in \cite[Remark 2.2]{CohenMorra}

We now state how to obtain $S_3$-cubic extensions with a fixed quadratic resolvent field. The following lemma is the content of Proposition 2.7 and Lemma 2.8 in \cite{CohenMorra}.

\begin{lemma}\label{lem:S3byQuad}
Let $S_3(F_z)$ be the 3-Selmer group of $F_z$, i.e. the group $\{u \in F_z \mid u\mathcal{O}_{F_z}=\mathfrak{q}^3\}/(L^\times)^3$
\begin{enumerate}
\item There exists a bijection between isomorphism classes of $S_3$-cubic extensions $K /k$ with given quadratic resolvent field $F/k$ and equivalence classes of triples $(a_0, a_1, u)$ modulo the equivalence relation $(a_0, a_1, u) \sim (a_1,a_0,1/u)$, where $a_0, a_1$, and $u$ are as follows:
\begin{itemize}
\item The $a_i$ are coprime integral square-free ideals of $F_z$ such that $a_0a_1^2 \in Cl(F_z)^3$ and $a_0a_1^2 \in (I/I^3)[T]$, where $I$ is the group of fractional ideals of $F_z$.
\item $u \in S_3(F_z)[T]$, and $u \neq 1$ when $a_0 =a_1 =Z_L$.
\end{itemize}
\item The condition $a_0a_1^2 \in (I/I^3)[T]$ is equivalent to: $a_1 = \tau_2(a_0)$ in case (i); $a_0=\tau_2(a_0)$ and $a_1 = \tau_2(a_1)$ in case (ii); $a_1=\tau(a_0)=\tau_2(a_0)$ in case (iii).
\item If $(a_0 , a_1)$ is a pair of ideals satisfying the above there exist an ideal $q_0$ and an element $\alpha_0$ of $F_z$ such that $a_0 a^2_1 q^3_0 = \alpha_0 \mathcal{O}_{F_z}$ with $\alpha_0 \in (F_z^*/(F_z^*)^3)[T]$. The cubic extensions $K/k$ corresponding to such a pair $(a_0,a_1)$ are given as follows: for any $u \in S_3(F_z)[T ]$ the extension is the cubic subextension of $N_z = F_z(\sqrt[3]{\alpha_0  u} )$ (for any lift $u$ of $\bar u$).
\end{enumerate}
\end{lemma}

\begin{remark}\label{rmk:allcases} Here we briefly summarize how to construct $S_3$-cubic extensions $K/k$ with a fixed quadratic resolvent $F/k$ following from Lemma \ref{lem:S3byQuad} applied to conditions (i), (ii), and (iii) listed above. Given $F/k$, we let $H$ be the Hilbert class field of $F/k$ and let $H_3$ be the extension intermediate to $H$ and $F$ such that $\mathrm{Gal}(H_3/F) \simeq \mathrm{Cl}(F)/\mathrm{Cl}(F)^3$. 
\begin{itemize}
\item[\textbf{Case (A)}] \textit{Constructing $S_3$-cubics following condition (i).}\label{ex:S3byA}
Let $k$ be a number field with $\zeta_3 \in k$ and let $F/k$ be a quadratic extension. Let $H$ be the Hilbert class field of $F$ and let $H_3$ be the extension intermediate to $H$ and $F$ such that $\mathrm{Gal}(H_3/F) \simeq \mathrm{Cl}(F)/\mathrm{Cl}(F)^3$. Let $p$ be a prime of $k$ such that $\mathrm{Frob}_p(H_3/k)=1$. Suppose that $p$ splits completely into two primes $\mathfrak{p}_0$ and $\mathfrak{p}_1$ in $F$ and both are cubes in $\mathrm{Cl}(F)$. Note also that $\tau_2$, the generator of $\mathrm{Gal}(F/k)$, exchanges $\mathfrak{p}_0$ and $\mathfrak{p}_1$. Now, following Lemma \ref{lem:S3byQuad}, there is an ideal $\mathfrak{q}$ of $F$ and element $\alpha \in F$ such that $\mathfrak{p}_0 \mathfrak{p}_1^2 \mathfrak{q}^3 = (\alpha)$. Finally, the extension $F(\sqrt[3]{\alpha})$ is a $C_3$-extension of $F$ and a Galois $S_3$-extension of $k$. Its cubic subextension is an $S_3$-cubic extension of $k$.

Likewise, if $p^{(1)},...,p^{(m)}$ are $m$ distinct primes of $k$ for which $\mathrm{Frob}_{p^{(j)}} (H_3/k)=1$, $1 \leq j \leq m$, and each splits as $\mathfrak{p}_0^{(j)}\mathfrak{p}_1^{(j)}$ in $F$, then setting $\mathfrak{a}_0=\prod_j \mathfrak{p}_0^{(j)}$ and $\mathfrak{a}_0=\prod_j \mathfrak{p}_1^{(j)}$, there are $\alpha \in F$ and $\mathfrak{q}$ so $\mathfrak{a}_0 \mathfrak{a}_1^2 \mathfrak{q}^3 = (\alpha)$ and we can obtain $S_3$-cubic extension as the cubic subextension of $F(\sqrt[3]{\alpha})$.

\textit{Suppose further that $H_3=F$}. Then, up to a choice of an inclusion of lifts of 3-Selmer units, this is a complete description of $S_3$-cubic extensions of $k$ with fixed quadratic resolvent $F/k$.

\item[\textbf{Case (B)}] \textit{Constructing $S_3$-cubics following condition  (ii).}\label{ex:S3byB}
We will work over the rational numbers. Let $F=\mathbb{Q}(\zeta_3)$ and note $\mathrm{Cl}(F)$ is trivial, In the notation of the previous example, $H_3=F$. 
\begin{itemize}
    \item Let $p$ be a rational prime which is inert in $F$; say $p\mathcal{O}_F = \mathfrak{p}$.  The prime ideal $\mathfrak{p}$ is fixed by $\tau_2$ and its ideal class is in $\mathrm{Cl}(F)^3$. By Lemma \ref{lem:S3byQuad}, any such prime $p$ gives an $S_3$-cubic extension as the appropriate subextension of $F(\sqrt[3]{p})$.
    \item Let $p$ be a rational prime which splits completely as $p\mathcal{O}_F = \mathfrak{p}\mathfrak{p}'$ in $F$. Let $\mathfrak{a}_0=\mathfrak{p}$ and $\mathfrak{a}_1=\mathfrak{p}'$. Note that $\mathfrak{a}_0$ and $\mathfrak{a}_1$ are exchanged under the action of $\tau_2$ and satisfy the the conditions of Lemma \ref{lem:S3byQuad}. In particular, the ideal class of $\mathfrak{a}_0\mathfrak{a}_1^2$ lies in $\mathrm{Cl}(F)^3$. Any such prime $p$ also gives an $S_3$-cubic extension as the appropriate subextension of $F(\sqrt[3]{p})$.
\end{itemize}
The same construction works for a base field $k$ for which $\zeta_3 \notin k$ and $F=k(\zeta_3)$ subject to the constraint that $[H_3:F]=1$, where, as above, $H_3$ is the sub-extension of the Hilbert class field of $F$ over $k$ whose Galois group is the cube of the relative class group of $F/k$.

\item[\textbf{Case (C)}] \textit{Constructing $S_3$-cubics following condition (iii).}\label{ex:S3byC}
Let $k$ be a number field and let $F/k$ be a quadratic extension so that $\zeta_3 \notin F$. Set $F_z = F(\zeta_3)$. As in the previous example, let $H_3$ denote the sub-extension of the Hilbert class field of $F_z$ so that  $\mathrm{Gal}(H_3/F_z) \simeq \mathrm{Cl}(F_z)^3$.

Let $p$ be a prime of $k$ so that $\mathrm{Frob}_{p}(H_3/k)=1$. Then $p$ splits completely into two primes in $F$, i.e. $p\mathcal{O}_{F}=\mathfrak{p}\mathfrak{p}'$. Similarly, $\mathfrak{p}$ and $\mathfrak{p}'$ split completely in $F_z$, say as $\mathfrak{p}\mathcal{O}_{F_z} = \mathfrak{P}_0\mathfrak{P}_1$ and $\mathfrak{p}'\mathcal{O}_{F_z} = \mathfrak{P}_0'\mathfrak{P}_1'$. Set $\mathfrak{a}_0=\mathfrak{P}_0\mathfrak{P}_0'$ and $\mathfrak{a}_1=\mathfrak{P}_1\mathfrak{P}_1'$ and note that $\tau(\mathfrak{a}_0)=\tau_2(\mathfrak{a}_0) = \mathfrak{a}_1$. Again following Lemma \ref{lem:S3byQuad}, let $\alpha_0 \in F_z$ be the element so that $\mathfrak{a}_0\mathfrak{a_1}_2 \mathfrak{q}_0^3 = \alpha_0\mathcal{O}_{F_z}$ for some ideal $\mathfrak{q}_0$. As before, a corresponding $S_3$-cubic extension of $k$ is the cubic subfield of $F_z(\sqrt[3]{\alpha_0})$. 

As in case (A), one could take multiple primes $p_i$ of $K$ so that $\mathrm{Frob}_{p_i}(H_3/k)=1$, and repeat the construction with the product of primes of $F_z$ above $p$.

\end{itemize}
\end{remark}

\begin{remark}
    Note that case \textbf{(B)} does not construct all of the $S_3$-cubic extensions implied by condition (ii) and Lemma \ref{lem:S3byQuad}. However, we need the explicit description of the primes in the parameterization afforded by the additional condition that $H_3 = F$.
\end{remark}

\begin{definition}\label{def:admissible}
We call a quadratic extension $F$ of a base field $k$ \textit{admissible} if they jointly meet the three rightmost conditions of any row in the following table.
\begin{center}
    \begin{tabular}{|c||c|c|c|c|}
    \hline
    Our Param. & Cohen-Morra's Cond. & $\zeta_3 \in k$? & $\zeta_3 \in F$? & Further Conditions on $F/k$? \\
    \hline
    \hline
    Case \textbf{(A)} & (i) & Yes & Yes & None \\ \hline
    Case \textbf{(B)} & (ii) &  No & Yes, i.e. $F=k(\zeta_3)$ & $H_3 = F$ \\
    \hline
    Case \textbf{(C)} & (iii) & No & No & None  \\
    \hline
    \end{tabular}
\end{center}
As before, if $H$ is the Hilbert class field of $F/k$, we define $H_3$ to be the extension intermediate to $H$ and $F$ such that $\mathrm{Gal}(H_3/F) \simeq \mathrm{Cl}(F)/\mathrm{Cl}(F)^3$.
\end{definition}

\section{Measuring Rank Growth in \texorpdfstring{$S_3$}{S3} Extensions}\label{sec:measure-rank-growth}
We now need to understand how the rank of an elliptic curve $E/k$ changes upon base extension  to a number field $L/k$ and how rank grows in various subextensions between $L$ and $k$. We will then specialize to the case when $L/k$ is a Galois $S_3$-extension, as in the previous section.

To get the desired relations, we use the results of \cite{DD}. Let $L/k$ be a Galois extension of number fields and set $G = \mathrm{Gal}(L/k)$. Fix an elliptic curve $E/k$ and let $\chi_L$ be the character of the complex representation of $E(L) \otimes \mathbb{C}$. 

For any subgroup $H \leq G$, write $\mathbf{1}_H$ for the trivial character on $H$. Now, for any $K/k$ fixed by $H \leq G$, we have 
$$\mathrm{rk}(E/K) = \langle \mathrm{Ind}_H^G \mathbf{1}_H , \chi_L \rangle.$$

Let $K_i/k$ and $K'_j/k$ be subextensions of $L$ fixed by subgroups $H_i \leq G$ and $H'_j \leq G$, respectively. If, as complex representations of $G$, we have 
$$\bigoplus_i \mathrm{Ind}_{H_i}^G \mathbf{1}_{H_i} = \bigoplus_j \mathrm{Ind}_{H_j}^G \mathbf{1}_{H'_j},$$
then 
$$\sum_i \mathrm{rk}(E(K_i)) = \sum_j \mathrm{rk}(E(K'_j)).$$
Letting $L/k$ be the Galois closure of an $S_3$-cubic extension of $k$, a group theory computation delivers the following lemma.

\begin{lemma}\label{lem:S3rankgrowth}
    Let $K/k$ be a an $S_3$-cubic extension with Galois closure $\tilde K/k$ and quadratic resolvent $F/k$. If $E$ is an elliptic curve over $k$, then
    \begin{equation}
        \mathrm{rk}(E( K)) - \mathrm{rk}(E( k)) = \frac{1}{2}\left(\mathrm{rk}(E(\tilde K))-\mathrm{rk}(E( F))  \right). 
    \end{equation}
\end{lemma}

\begin{remark}
    In the previous lemma, $\tilde K/F$ is a $\mathbb{Z}/3\mathbb{Z}$-cubic extension, and $\mathrm{rk}(E(\tilde K))-\mathrm{rk}(E( F))$ is the rank of the 2-dimensional abelian variety obtained by taking the kernel of the Weil restriction map $\Res^{\tilde K}_F E \to E$. We will consider these and related varieties in the following sections.
\end{remark}

\subsection{Useful abelian varieties and their prime Selmer groups}
Let $E$ be an elliptic curve over $k$. In the spirit of \cite{MR07}, we construct a 4-dimensional abelian variety whose algebraic rank of $k$-rational points can be utilized to understand rank growth of $E$ with respect to an $S_3$-cubic $K/k$. As in previous sections, $K/k$  denotes an $S_3$-cubic extension with Galois closure $\tilde K/k$, which in turn has a unique quadratic subfield $F/k$, i.e. its quadratic resolvent. 

\begin{definition} \label{def:4-dimAV}
    Define a $4$-dimensional abelian variety over $k$ by
    \begin{equation}
        B_{K/k} := \text{Ker} \left(\text{N}_F^{\tilde{K}}: \text{Res}_k^{\tilde{K}} E \to \text{Res}_k^F E \right),
    \end{equation}
    where $\text{N}_F^{\tilde{K}}: \tilde K \to F$ is the field norm map.
\end{definition}
The variety $B_{K/k}$ depends on $F$. Also, note $B_{K/k}$ corresponds to the isotypic component of the 2-dimensional standard representation of $S_3$ of the 6-dimensional abelian variety $\text{Res}_k^{\tilde{K}} E$ over $k$. Hence, it is the case that $\mathrm{rk}(B_{K/k}(k)) \equiv 0 \text{ mod } 2$. 

For $F=k(\sqrt{\alpha})$, let $E_F$ denote the quadratic twist of $E/k$ by $\alpha$. Similar to the above, one obtains the relation
\begin{equation}
    \mathrm{rk}(E(\tilde{K})) = \mathrm{rk}(E(K)) + \mathrm{rk}(E_F(k)) + \mathrm{rk}(B_{K/k}(k)).
\end{equation}
Above, each summand on the right hand side corresponds to isotypic components given by the trivial representation, the 1-dimensional sign representation, and the 2-dimensional standard representation of $S_3$. 
Using Lemma \ref{lem:S3rankgrowth}, we obtain
\begin{equation}
    \mathrm{rk}(E(K)) - \mathrm{rk}(E(k)) = \frac{1}{2} \mathrm{rk}(B_{K/k}(k)).
\end{equation}

Pick an order $3$ element $\sigma_K \in \text{Gal}(\tilde{K}/F)$. Because $\text{Gal}(\tilde{K}/F)$ is a unique normal subgroup of $\text{Gal}(\tilde{K}/k)$, the element $\sigma_K$ is an endomorphism of $B_{K/k}$. In lieu of \cite[Proposition 4.1]{MR07}, we apply \cite[Proposition 4.2.4]{SWPthesis} to give the following description of $1-\sigma_K$ torsion subgroup of $B_{K/k}$.
\begin{proposition} \label{prop:torsion}
    There exists a \emph{$\text{Gal}(\overline{k}/k)$}-equivariant isomorphism
    \begin{equation*}
        B_{K/k}[1-\sigma_K] \cong \left( \emph{\text{Res}}_k^F E \right)[3] \cong \emph{\text{Ind}}_{\emph{\text{Gal}}(\overline{k}/F)}^{\emph{\text{Gal}}(\overline{k}/k)} E[3]
    \end{equation*}
\end{proposition}
\begin{proof}
    The proposition can be deduced from taking Weil restriction $\text{Res}_k^F$ on both sides of the Galois-equivariant isomorphism stated in \cite[Proposition 4.1]{MR07}. To make the paper self contained, we also present a different strategy of the proof.\footnote{The authors thank Jordan Ellenberg for proposing a generalized version of this construction.} We first propose an equivalent construction of $B_{K/k}$. Consider the coordinate addition map $\mathbb{Z}^3 \to \mathbb{Z}$ by $(x_i)_{i=1}^3 \mapsto \sum_{i=1}^3 x_i$, and denote its kernel by $S(3)$. We then have an exact sequence 
    \begin{align}
        \begin{split}
            0 \to S(3) \to \mathbb{Z}^3 &\to \mathbb{Z} \to 0. 
        \end{split}
    \end{align}
    Given a fixed elliptic curve $E$ over a global field $k$, let $B_E$ be the 4-dimensional abelian variety over $k$ defined as
    \begin{equation}
        B_E := \text{Hom}_{\mathbb{Z}}\left(\text{Hom}_{\mathbb{Z}}(S(3),S(3)),E\right).
    \end{equation}
    Note that $B_E$ is an $S_3$-bimodule, and $\mathrm{End}_\mathbb{Z}(B_E) \supset \mathbb{Z}[S_3]$. Let $\psi: \text{Gal}(\overline{k}/k) \to \text{End}_\mathbb{Z}(B_E)$ be a morphism which factors as
    \begin{equation}
        \psi: \text{Gal}(\overline{k}/k) \to \text{Gal}(\tilde{K}/k) \cong S_3 \to \text{End}_{\mathbb{Z}}(B_E).
    \end{equation}
    By construction, the twist of $B_E$ by $\psi$, denoted as $B_E^\psi$, is isomorphic to $B_{K/k}$ as abelian varieties over $k$. Note that given any $\tau \in \text{Gal}(\overline{k}/k)$, the element $\psi(\tau)$ acts on $B_E$ by multiplying the elements of $B_E$ on the left via matrix multiplication.

    As a $2 \times 2$ matrix over $\mathbb{Z}$, the representations of $\sigma_K \in \text{Gal}(\tilde{K}/F)$ and $1-\sigma_K$ are given by
    \begin{equation}
        \sigma_K := \begin{bmatrix}
        -1 & -1 \\
        1 & 0
        \end{bmatrix}, \quad \text{and} \quad  1 - \sigma_K := \begin{bmatrix}
        2 & 1 \\
        -1 & 1
        \end{bmatrix}.
    \end{equation}
    The left exactness of the $\text{Hom}_{\mathbb{Z}}(\cdot, E)$ functor implies
    \begin{align}
        \begin{split}
            B_E^{\psi}[1-\sigma_K] &\cong \text{Hom}_{\mathbb{Z}} \left( \frac{\text{Hom}_{\mathbb{Z}}(S(3),S(3))}{(1-\sigma_K)\text{Hom}_{\mathbb{Z}}(S(3),S(3))}, E\right)
        \end{split}
    \end{align}
    as $\text{Gal}(\overline{K}/K)$- modules,
    where $\psi$ induces the action of $\text{Gal}(\overline{K}/K)$ by left multiplication with elements in $S_3 \cong \text{SL}_2(\mathbb{F}_2)$. But notice that 
    \begin{align}
        \begin{split}
            \frac{\text{Hom}_{\mathbb{Z}}(S(3),S(3))}{(1-\sigma_K)\text{Hom}_{\mathbb{Z}}(S(3),S(3))} &\cong (\mathbb{Z}/3\mathbb{Z})^{\oplus 2}.
        \end{split}
    \end{align}
    Here, the non-trivial order-3 element $\sigma_K \in \text{Gal}(\tilde{K}/k)$ acts trivially on the module $(\mathbb{Z}/3\mathbb{Z})^{\oplus 2}$, whereas $\tau_2 \in \text{Gal}(F/k)$ acts as an involution on the module $(\mathbb{Z}/3\mathbb{Z})^{\oplus 2}$ by swapping the coordinates. This implies that
    \begin{equation}
       \text{Hom}_{\mathbb{Z}} \left( \frac{\text{Hom}_{\mathbb{Z}}(S(3),S(3))}{(1-\sigma_K)\text{Hom}_{\mathbb{Z}}(S(3),S(3))}, E\right) \cong \text{Hom}_{\mathbb{Z}}((\mathbb{Z}/3\mathbb{Z})^{\oplus 2}, E) \cong (\text{Res}_k^F E)[3].
    \end{equation}
    Because given any $\tau \in \text{Gal}(\overline{k}/k)$, $\psi$ acts as an automorphism on $(\text{Res}_k^F E)[3]$, we obtain the desired $\text{Gal}(\overline{k}/k)$-equivariant isomorphism $B_{K/k}[1-\sigma_K] \cong (\text{Res}_k^F E)[3]$.
\end{proof}

\begin{remark}
    The multiplication by $3$ map $\times 3: B_{K/k} \to B_{K/k}$ can be factored as the composition $(1 - \sigma_K) \circ (1-\sigma_K)$. This can be seen from the fact that the minimal polynomial of $1-\sigma_K$ over $\mathbb{Z}$ is given by $\lambda^2 + 3(\lambda + 1) = 0$.
\end{remark}

Using the isogeny $1-\sigma_K: B_{K/k} \to B_{K/k}$, one can obtain the following short exact sequence of Galois cohomology groups over $k$:
\begin{equation*}
    \begin{tikzcd}
        0 \arrow[r] & \frac{B_{K/k}(k)}{(1-\sigma_K)B_{K/k}(k)} \arrow[r] \arrow[d] &H^1(k, B_{K/k}[1-\sigma_K]) \arrow[r] \arrow[d, "\prod_v \text{res}_v"] &H^1(k, B_{K/k})[1-\sigma_K] \arrow[r] \arrow[d] &0 \\
        0 \arrow[r] & \prod_{v} \frac{B_{K/k}(k_v)}{(1-\sigma_K)B_{K/k}(k_v)} \arrow[r, "\prod_v \delta_v"] & \prod_v H^1(k_v, B_{K/k}[1-\sigma_K]) \arrow[r]  & \prod_v H^1(k_v, B_{K/k})[1-\sigma_K] \arrow[r]  &0.
    \end{tikzcd}
\end{equation*}
As in Section \ref{sec:Selmer} and Definition \ref{def:selmergroup}, the $(1-\sigma_K)$-Selmer group of $B_{K/k}$ is given by
\begin{equation}
    \Sel_{1-\sigma_K}(B_{K/k}/k) := \{c \in H^1(k,B_{K/k}[1-\sigma_K]) \; : \; \prod_v \text{res}_v(c) \in \prod_v H^1_f(k_v, B_{K/k}[1-\sigma_K]) \}.
\end{equation}

Proposition \ref{prop:torsion} implies,  for any collections of $S_3$-cubics extensions $\{K/k\}$ with \textit{fixed}  quadratic resolvent field $F/K$, that
\begin{equation}
    \Sel_{1-\sigma_K}(B_{K/k}/k) \subset H^1(k, (\text{Res}_k^F E)[3]) \cong H^1(F, E[3]),
\end{equation}
where the isomorphism follows from Shapiro's lemma. A standard argument using Galois cohomology groups demonstrate that $\Sel_{1-\sigma_K}(B_{K/k}/k)$ is a finite dimensional $\mathbb{F}_3$-vector space. Furthermore, one has the relation
\begin{equation}
    \text{rk}(E(K)) - \text{rk}(E(k)) \leq \dim_{\mathbb{F}_3} \Sel_{1-\sigma_K}(B_{K/k}/k).
\end{equation}
This follows the following equation by $2$, which is obtainable from \cite[Proposition 6.3]{MR07}:
\begin{equation}
    \text{corank}_{\mathbb{Z}_3[\zeta_3]} \Sel_{3^\infty}(B_{K/k}/k) \leq \dim_{\mathbb{F}_3} \Sel_{1-\sigma_K}(B_{K/k}/k).
\end{equation}

\subsection{Local conditions defining the desired prime Selmer groups}

The local conditions that define $\Sel_{1-\sigma_K}(B_{K/k}/k)$ depend on the decomposition of prime ideals of $k$ over in quadratic resolvent field $F/k$. To identify the differences in such local conditions, we introduce the notion of coordinate-wise maximal isotropic subspaces of a direct sum of quadratic vector spaces over finite fields. The results stated here are for our particular $S_3$ extensions; a flexible, general version can be found in the PhD thesis of the second author \cite[Section 4.3]{SWPthesis}.

\begin{definition}
    Let $q: V \times V \to \mathbb{F}_p$ be a non-degenerate quadratic form on a finite dimensional $\mathbb{F}_p$- vector space $V$. For any $n \geq 1$, one obtains a non-degenerate quadratic form $q^{\oplus n}: V^{\oplus n} \times V^{\oplus n} \to \mathbb{F}_p^{\oplus n}$. For each $1 \leq i \leq n$, denote by $\pi_i: V^{\oplus n} \to V$ the projection morphism to the $i$-th coordinate. We say that a subspace $W \subset V^{\oplus n}$ is \textit{coordinate-wise Lagrangian} if the following two conditions hold for every $1 \leq i \leq n$.
    \begin{enumerate}
        \item The subspace $\pi_i(W) \subset V$ is a maximal isotropic subspace of $V$ with respect to the quadratic form $q$.
        \item The quadratic form $q$ over $V$ is trivial over $\pi_i(W)$.
    \end{enumerate}
\end{definition}

We identify the local conditions which defines $\Sel_{1-\sigma_3}(B_{K/k}/k)$ at primes $v$ of $k$ depends on the splitting behavior of $v$ with respect to the quadratic extension $F/k$. 

\begin{example}[Split primes over $F/k$] \label{ex:split}
Suppose the place $v$ of $k$ is unramified over $F/k$ and factorizes over $\mathcal{O}_F$ as $(v) = \mathfrak{p}_1 \mathfrak{p}_2$. Assume further that $v$ is not a place of bad reduction of $E$. Then $F_v \cong F_{\mathfrak{p}_1} \oplus 
F_{\mathfrak{p}_2} \cong k_v \oplus k_v$ and
\begin{equation}
    H^1(k_v, (\text{Res}_k^F E)[3]) \cong H^1(k_v, E[3])^{\oplus 2}.
\end{equation}
The Weil pairing $E[3] \times E[3] \to \mu_3$ induces a coordinate-wise symmetric pairing
\begin{equation}
    H^1(k_v, (\text{Res}_k^F E)[3]) \times H^1(k_v, (\text{Res}_k^F E)[3]) \to \mu_3^{\oplus 2}.
\end{equation}
Denote by $q_v := q_{\mathfrak{p}_1} \oplus q_{\mathfrak{p}_2}:H^1(k_v, (\text{Res}_k^F E)[3]) \to \mu_3^{\oplus 2}$ the quadratic form induced from the coordinate-wise symmetric pairing. If we base change the fixed elliptic curve $E$ to be defined over $F$, then we can define the 2-dimensional abelian variety over $F$, denoted as $A_{\tilde{K}/F}$, to be given by
\begin{equation}
    A_{\tilde{K}/F} := \text{Ker} \left(N_F^{\tilde{K}}: \text{Res}_F^{\tilde{K}} E \to E \right).
\end{equation}
Then it follows from \cite[Section 4]{MR07} and \cite[Section 4]{PR12} that both images of the morphisms
\begin{align}
\begin{split}
    \text{loc}_v &: H^1(k, \text{Res}_k^F E[3]) \cong H^1(F, E[3]) \to H^1(F_{\mathfrak{p}_1}, E[3]) \oplus H^1(F_{\mathfrak{p}_2}, E[3]) \cong H^1(k_v, E[3])^{\oplus 2} \\ \text{and} \quad
    \delta_v &: \frac{B_{K/k}(k_v)}{(1-\sigma_K)B_{K/k}(k_v)} \cong  \frac{A_{\tilde{K}/F}(F_{\mathfrak{p}_1})}{(1-\sigma_K)A_{\tilde{K}/F}(F_{\mathfrak{p}_1})} \oplus \frac{A_{\tilde{K}/F}(F_{\mathfrak{p}_2})}{(1-\sigma_K)A_{\tilde{K}/F}(F_{\mathfrak{p}_2})} \to H^1(k_v, E[3])^{\oplus 2}
    \end{split}
    \end{align}
are coordinate-wise Lagrangian subspaces of $H^1(K_v, (\text{Res}_K^M E)[3])$. That is, each coordinate is a maximal isotropic subspace $V$ of $H^1(K_v, E[3])$ such that $q_v|_V = 0$. The first map $\text{loc}_v$ is a direct product of two localization maps $\text{loc}_{\mathfrak{p}_1} \oplus \text{loc}_{\mathfrak{p}_2}$; the second map $\delta_v$ is a direct product of two local Kummer maps $\delta_{\mathfrak{p}_1} \oplus \delta_{\mathfrak{p}_2}$. Since $k_v \cong F_{\mathfrak{p}_1} \cong F_{\mathfrak{p}_2}$, it follows that the images of the local Kummer maps $\delta_{\mathfrak{p}_1}$ and $\delta_{\mathfrak{p}_2}$ have to be identical to each other. In particular, if $v$ is a place of good reduction of $B_{K/k}$, then the unramified cohomology group of $H^1(k_v, \text{Res}_k^F E[3])$ is given by
\begin{equation}
    H^1_f(k_v, B_{K/k}[1-\sigma_K]) = H^1_{ur}(k_v, \text{Res}_k^F E[3]) = H^1_{ur}(k_v, E[3])^{\oplus 2}.
\end{equation}
We note that \cite[Section 3]{KMR14} shows that the number of ramified coordinate-wise Lagrangian subspaces of $H^1(k_v,\text{Res}_k^F E[3])$ can be characterized as follows. Here, any Lagrangian subspace $V \subset H^1(k_v,\text{Res}_k^F E[3])$ is ramified if $V \neq H^1_{ur}(k_v, \text{Res}_k^F E[3])$. We hence have
\begin{align}
    \begin{split}
        \dim_{\mathbb{F}_3} E[3](k_v) = i \iff \# \{\text{Coordinate-wise Ramified Lagrangian Subspaces} \} = \lceil 3^{2(i-1)} \rceil.
    \end{split}
\end{align}
Among the ramified coordinate-wise Lagrangian subspaces of $H^1(k_v, E[3])^{\oplus 2}$, only $3^{i-1}$ of them can occur as the restricted cohomology group $H^1_f(k_v, B_{K/k}[1-\sigma_K])$.
\end{example}

\begin{example}[Inert or ramified primes over $F/k$] \label{ex:inert}
Now we suppose the place $v$ of $k$ is inert or ramified over $F/k$. Assume further that $v$ is not a place of bad reduction of $E$. Then $F_v$ is an unramified (or respectively ramified) quadratic extension of $k_v$. Shapiro's lemma implies that
    \begin{equation}
        H^1(k_v, (\text{Res}_k^F E)[3]) \cong H^1(F_v, E[3]),
    \end{equation}
    and the Weil pairing $E[3] \times E[3] \to \mu_3$ over $F_v$ (not over $k_v$) induces a symmetric pairing
    \begin{equation}
        H^1(k_v, \text{Res}_k^F E[3]) \times H^1(k_v, \text{Res}_k^F E[3]) \to \mu_3.
    \end{equation}
    As before, denote by $q_v: H^1(k_v, \text{Res}_k^F E[3])$ the quadratic form induced from the coordinate-wise symmetric pairing. By \cite[Section 4]{MR07} and \cite[Section 4]{PR12}, we obtain that both images of the morphisms
    \begin{align}
    \begin{split}
        \text{loc}_v &: H^1(k, \text{Res}_k^F E[3]) \cong H^1(F, E[3]) \to H^1(F_v, E[3]) \\
       \text{and} \quad \delta_v &: \frac{B_{K/k}(k_v)}{(1-\sigma_K)B_{K/k}(k_v)} \cong \frac{A_{\tilde{K}/F}(F_v)}{(1-\sigma_K)A_{\tilde{K}/F}(F_v)} \to H^1(F_v, E[3])
    \end{split}
    \end{align}
    are Lagrangian subspaces of $H^1(k_v, (\text{Res}_k^F E)[3])$. As shown in \cite[Section 3]{KMR14}, the number of ramified coordinate-wise Lagrangian subspaces of $H^1(k_v,\text{Res}_k^F E[3])$ can be characterized as follows:
    \begin{align}
        \begin{split}
            \dim_{\mathbb{F}_3} E[3](k_v) = i \iff \# \{\text{Coordinate-wise Ramified Lagrangian Subspaces} \} = \lceil 3^{i-1} \rceil.
        \end{split}
    \end{align}
    All of these subspaces arise from restricted cohomology groups $H^1_f(k_v, B_{K/k}[1-\sigma_K])$ for some $S_3$-cubic extension $K/k$ and its associated 4-dimensional abelian variety $B_{K/k}$.
\end{example}

\begin{remark} \label{remark:local-conditions-Kummer}
    From both examples, we observe that regardless of whether the place $v$ of $k$ is split, inert, or ramified over $F/k$, as long as $v$ is a place of good reduction of $E$ over $k$, we obtain
    \begin{align}
    \begin{split}
        & \dim_{\mathbb{F}_3} E[3](k_v) = i \iff \\
        & \# \{\text{Coordinate-wise Ramified Lagrangian Subspaces of form } H^1_f(k_v, B_{K/k}[1-\sigma_K]) \} = \lceil 3^{i-1} \rceil.
    \end{split}
    \end{align}
    The above equivalence relation can be formulated as follows.
    \begin{itemize}
        \item If $i = 0$, then $H^1(k_v, \text{Res}_k^F E[3])$ is trivial, and changes in local conditions at such places $v$ do not affect the cardinality of $\Sel_{1-\sigma_K}(B_{K/k}/k)$.
        \item If $i = 1$, then $H^1(k_v, \text{Res}_k^F E[3])$ has a unique ramified coordinate-wise Lagrangian subspace that is equal to $H^1_f(k_v, B_{K/k}[1-\sigma_K])$ for some $K/k$.
        \item If $i = 2$, then $H^1(k_v, \text{Res}_k^F E[3])$ has three ramified coordinate-wise Lagrangian subspaces that are equal to $H^1_f(k_v, B_{K/k}[1-\sigma_K])$ for some $K/k$, each of which arises from three ramified order-3 local characters $\omega: \text{Gal}(\overline{k_v}/k_v) \to \mathbb{Z}/3\mathbb{Z}$ over the local field $k_v$.
    \end{itemize}
\end{remark}

We also obtain an analogue of Poonen-Rains heuristics which states that the intersection of two maximal isotropic subspaces characterizes the dimension of $(1-\sigma_K)$-Selmer group of $B_{K/k}$.
\begin{proposition} \label{prop:random-matrix-model}
    Consider the following diagram:
    \begin{equation}
        \begin{tikzcd}
            & H^1(k, (\emph{\text{Res}}_k^F E)[3]) \arrow{d}{} \\
\prod_v \frac{B_{K/k}(k_v)}{(1-\sigma_K)B_{K/k}(k_v)} \arrow{r}{} & \prod_{v \text{ place of } k} H^1(k_v, (\emph{\text{Res}}_k^F E)[3]).
        \end{tikzcd}
    \end{equation}
    \begin{enumerate}
        \item For each place $v$ of $k$, the images of the horizontal and vertical maps are coordinate-wise Lagrangian subspaces term with respect to the quadratic form $q_v$.
        \item The intersection of the images of the horizontal and vertical maps are isomorphic to $\Sel_{1-\sigma_K}(B_{K/k} / k)$.
    \end{enumerate}
\end{proposition}
\begin{proof}
The fact that the image of the horizontal map is a maximal isotropic subspace of the lower-right term follows from \cite[Proposition 4.4]{MR07}. The fact that the vertical map is a maximal isotropic subspace can be obtained from adapting the proof of \cite[Theorem 4.13]{PR12}, where the result follows from the 9-term Poitou-Tate sequence.
\end{proof}

\section{Markov Chains and Ranks}\label{sec:markov-chains-ranks}

\subsection{Preliminary results}

Recall from the previous section that given a cyclic $\mathbb{Z}/3\mathbb{Z}$-extension $K/k$, one can define the $2$-dimensional abelian variety
\begin{equation*}
    A_{K/k} := \text{Ker}(N_k^K: \text{Res}_k^K E \to E)
\end{equation*}
and using the cyclic generator $\sigma_K \in \text{Gal}(K/k)$, the Selmer group $\Sel_{1-\sigma_K}(A_{K/k}/k)$ governs the rank growth of $E$ with respect to the extension $K/k$. The dimensions of these Selmer groups can be understood from constructing what are called the local Selmer structures. These are characterized as a collection of finite dimensional vector sub-spaces of $H^1(k, E[3])$ and are heavily utilized in \cite{MR07, KMR14}. The goal of this section is to introduce relevant notations and key results regarding local Selmer structures as outlined in \cite{KMR14}, which will serve as backbone for understanding the dimensions of $\Sel_{1-\sigma_K}(B_{K/k}/k)$ as $K/k$ varies over $S_3$-cubic extensions.

Before starting the next definition, we again note that $k$ is the base field, $F/k$ is a quadratic extension (and will be the quadratic resolvent of $S_3$-cubic extensions of $k$), and by assumption the elliptic curve $E$ is such that $k(E[3])$ and $F$ are linearly disjoint over $k$. 

\begin{definition} \label{def:DA_DB_DC}
    We define a set of square-free products of primes of $F$ corresponding to cases (A), (B), (C) with conditions on the primes as appearing in Remark \ref{rmk:allcases},
    \begin{itemize} 
    \item \textit{A set of primes for case (A).} $k$ is a number field such that $\zeta_3 \in k$, $F/k$ is a quadratic extension, $H_3$ is the subextension of the Hilbert class field of $F$ for which $\mathrm{Gal}(H_3/F) \simeq \mathrm{Cl}(F)^3$. Define \[   \mathcal{D}^A := \left\{\mathfrak{d} = \prod_{i=1}^n p_i\mathcal{O}_F : p_i~\text{a prime of $k$ such that}~ \mathrm{Frob}_{p_i}(FH_3/k)=1, \mathrm{Frob}_{p_i}(k(E[3])/k) \neq 1 \right\}. \]

    \item \textit{A set of primes for case (B)}. Here $k=\mathbb{Q}$, $F= \mathbb{Q}(\zeta_3)$, and so the Hilbert class field of $F$ is trivial.
 \[ \mathcal{D}^B := \left\{\mathfrak{d} = \prod_{i=1}^n p_i\mathcal{O}_F : p_i~\text{a prime of $k$ such that }  \mathrm{Frob}_{p_i}(k(E[3])/k) \neq 1 \right\}. \]

\item \textit{A set of primes for case (C)}. Here $k$ is any number field, $F/k$ is a quadratic extension for which $\zeta_3 \notin F$, $F_z = F(\zeta_3)$, and $H_3$ is the subextension of the Hilbert class field of $F_z$ for which $\mathrm{Gal}(H_3/F_z) \simeq \mathrm{Cl}(F_z)^3$. Define
  \[ \mathcal{D}^C := \left\{\mathfrak{d} = \prod_{i=1}^n p_i\mathcal{O}_F : p_i~\text{a prime of $k$ such that }  \mathrm{Frob}_{p_i}(k(E[3])/k) \neq 1 \text{ and } \mathrm{Frob}_{p_i}(H_3/k) = 1 \right\}. \]

\item Define $\mathcal{P}_i^k$ to be the set of places of $k$ such that $\dim_{\mathbb{F}_3} E[3](k_v) = i$.

\item Define $\mathcal{D}$ to be square-free products of primes supported on $\mathcal{P}_1^k \cup \mathcal{P}_2^k$.

\item We write $\mathcal{D}^*$ in the results that follow to indicate that any fixed choice of $\mathcal{D}, \mathcal{D}^A, \mathcal{D}^B$, or $\mathcal{D}^C$ may be taken in place of $\mathcal{D}^*$.
\end{itemize}
\end{definition}

\begin{remark}
    Results below are concerned with working out the local conditions of Selmer groups at places supported by elements of $\mathcal{D}^*$ and can be worked out for a choice $\mathcal{D}^*$ taken to be $\mathcal{D}, \mathcal{D}^A, \mathcal{D}^B$, or $\mathcal{D}^C$. Since we are concerned with the parameterizations of $S_3$-cubic extensions, we will only specialize to the cases that $\mathcal{D}^*$ is $\mathcal{D}^A, \mathcal{D}^B$, or $\mathcal{D}^C$. For the case $\mathcal{D}^* = \mathcal{D}$, see the final sections of \cite{KMR14}. 
\end{remark}

Next we define the sets of primes on which the ideals in $\mathcal{D}^*$ are supported.
\begin{definition}
    Fix a quadratic extension $F/k$. Define the following sets of primes,
    \[\mathcal{P}_i^{F,*} := \mathcal{P}^F_i \cap \mathcal{D}^*,\]
    i.e., the primes $\mathfrak{p}$ of $F$ in the support of elements of $\mathcal{D}^*$ for which the $\mathrm{Frob}_{\mathfrak{p}}$ on $E[3]$ has $3^{2-i}$ fixed points (the latter being equivalent to $\dim_{\mathbb{F}_3}E[3]/(\mathrm{Frob}_{\mathfrak{p}}-1) = i$).

    Further, define
     
    \[\mathcal{P}_i^{k,*} := \{ p = \mathfrak{p} \cap \mathcal{O}_k| \mathfrak{p} \in \mathcal{P}^F_i \cap \mathcal{D}^*\}\]
    to be the primes of $k$ below primes of $\mathcal{P}_i^{F,*}$.
\end{definition}
As with $\mathcal{D}^*$, we will express results below for the appropriate $\mathcal{P}_i^{k, *}$ which are applicable for $\mathcal{P}_i^{k, A}$, $\mathcal{P}_i^{k, B}$, $\mathcal{P}_i^{k, C}$.

\begin{definition}
    Let $v$ be a place of $k$, and let $E$ be a fixed elliptic curve over $k$. We introduce the following notations which are introduced in \cite{KMR14}.
    \begin{itemize}
        \item $\Sigma$: A finite set of places of $k$ containing places of bad reduction of $E$, places above $(3) \subset \mathbb{Z}$, all Archimedean places of $k$, and places dividing any one of the elements $u$ in $S_3(F_z)[T]$, the set of virtual 3-Selmer units of $F_z$. 
        \item $\mathfrak{d}$: A square-free product of places of $k$.
        \item $\Sigma(\mathfrak{d})$: A finite set of places of $k$ containing places dividing $\omega$ and places in $\Sigma$. 
        \item $\mathcal{C}(k)$: The set of characters
        \begin{equation}
            \mathcal{C}(k) := \text{Hom}(\text{Gal}(\overline{k}/k), \mu_3).
        \end{equation}
        \item $\mathcal{C}_{ram}(k)$: The set of ramified $\mu_3$-characters of $\text{Gal}(\overline{k}/k)$ in $\mathcal{C}(k)$.
        \item $k(E[3])$: The minimal Galois extension of $k$ over which all the $3$-torsion points of $E$ are rational. 
        \item $ k_{\mathfrak{d},\omega}$: The fixed field of $\bigcap_{c \in \Sel(E[3], \omega)_k} \text{Ker}(c: \text{Gal}(\overline{k}/k(E[3])) \to E[3])$, see \eqref{eq:Selmerfieldk} and \eqref{eq:SelmerfieldF}
        \item $\mathrm{Res}_{k(E[3])}$: The map given by the composition 
        $$H^1(k, E[3]) \to H^1(k(E[2]), E[3]) \overset{\simeq}{\to} \mathrm{Hom}(G_{k(E[3])}, E[2])^{\mathrm{Gal}(k(E[3])/k)}.$$
        \item $\Omega_\mathfrak{d}$: A set of Cartesian products of local characters
        \begin{equation}
            \Omega_\mathfrak{d} := \prod_{v \in \Sigma} \mathcal{C}(k_v) \times \prod_{v \mid \mathfrak{d}} \mathcal{C}_{ram}(k_v).
        \end{equation}
        \item $\Omega_\mathfrak{d}^{S}$: Given a subset $S \subset \prod_{v \in \Sigma} \mathcal{C}(k_v)$, it is a subset of Cartesian products of local characters
        \begin{equation}
            \Omega_\mathfrak{d}^S := S \times \prod_{v \mid \mathfrak{d}} \mathcal{C}_{ram}(k_v) \subset \Omega_\mathfrak{d}.
        \end{equation}
        \item $i_{\mathfrak{d}}: \Omega_\mathfrak{d} \to \prod_{v \text{ place of } k} \mathcal{C}(k_v)$: The inclusion map of characters such that $\omega_v$ is unramified if $v \not\in \Sigma(\mathfrak{d})$.
        \item $\eta_{\mathfrak{d},v}: \Omega_{\mathfrak{d}v} \to \Omega_{\mathfrak{d}}$: The projection map, if $\mathfrak{d}v \in \mathcal{D}$.
        \item $\omega_v$: A local character in $\mathcal{C}(k_v)$.
        \item $A_{K/k}^{\omega_v}$: A 2-dimensional abelian variety over $k_v$ obtained from twisting $A_{K/k}$ by a local character $\omega_v \in \mathcal{C}(k_v)$.
        \item $\mathcal{H}_{\omega_v}$: The restricted cohomology group obtained from the image of the local Kummer map induced from twisting by the local character $\omega_v \in \mathcal{C}(k_v)$:
        \begin{equation}
            \mathcal{H}_{\omega_v} := \text{Im} \left( \frac{A_{K/k}^{\omega_v}(k_v)}{(1-\sigma_K)A_{K/k}^{\omega_v}(k_v)} \to H^1(k_v, E[3]) \right).
        \end{equation}
        It is not hard to show that $\mathcal{H}_{\omega_v}$ is a Lagrangian subspace of $H^1(k_v, E[3])$, and that there exists a cyclic $\mathbb{Z}/3\mathbb{Z}$ extension $K/k$ such that $\mathcal{H}_{\omega_v} = H^1_f(k_v, A_{K/k}[1-\sigma_K])$. In particular if $v \in \Sigma(w)$, then $\mathcal{H}_{\omega_v}$ is coordinate-wise ramified.
        \item $\mathcal{H}_{\omega}$: A Cartesian product of restricted cohomology groups $\mathcal{H}_{\omega} := (\mathcal{H}_{\omega_v})_{v \text{ place of } k}$ obtained from $\omega \in i_{\mathfrak{d}}\left(\Omega_{\mathfrak{d}}\right)$.
    \end{itemize}
\end{definition}

\begin{definition}[Local Selmer structure]
\begin{enumerate}
    \item[]
    \item Let $\mathcal{H} := (\mathcal{H}_v)_{v \text{ place of } k}$ be a Cartesian product of Lagrangian subspaces $\mathcal{H}_v \subset H^1(k_v, E[3])$. The local Selmer structure of $E[3]$ associated to $\mathcal{H}$ is defined as
    \begin{equation}
        \Sel(E[3], \mathcal{H})_k := \{c \in H^1(k, E[3]) \; : \; \prod_v \text{res}_v(c) \in \prod_v \mathcal{H}_v\}.
    \end{equation}
    \item Given a choice of a square-free product $\mathfrak{d} \in \mathcal{D}$ and a Cartesian product of local characters $\omega \in i_{\mathfrak{d}}(\Omega_\mathfrak{d})$, the local Selmer group of $E[3]$ associated to $\omega$ is defined as
    \begin{equation}
        \Sel(E[3], \omega)_k := \Sel(E[3], \mathcal{H}_\omega)_k = \{c \in H^1(k, E[3]) \; : \; \prod_v \text{res}_v(c) \in \prod_v \mathcal{H}_{\omega_v}\}.
    \end{equation}
    \item Given a character $\omega \in i_{\mathfrak{d}}(\Omega_\mathfrak{d})$, we denote by $\text{rk}(\omega)$ the dimension
    \begin{equation}
        \text{rk}(\omega) := \dim_{\mathbb{F}_3} \Sel(E[3], \omega)_k.
    \end{equation}
\end{enumerate}
\end{definition}

\begin{definition}
    Suppose that $\mathfrak{d} \in \mathcal{D}$ and $\omega \in \Omega_{\mathfrak{d}}$. If $v \in \mathcal{P}_1 \cup \mathcal{P}_2$ and $v \nmid \mathfrak{d}$, we denote by
    \begin{equation}
        t_\omega(v) := \dim_{\mathbb{F}_3} \text{Im} \left( \text{loc}_v: \Sel(E[3], \omega)_k \to H^1_{ur}(k_v, E[3]) \right).
    \end{equation}
    where $\text{loc}_v$ denotes the localization map with respect to the place $v$.
\end{definition}

The differences between the dimensions of local Selmer structures constructed from two Cartesian product of Lagrangian subspaces $\mathcal{H}$ and $\mathcal{H}'$ that only differ at a single place $v$ can be obtained from the following proposition.
\begin{proposition}[Proposition 7.2 of \cite{KMR14}]
    Suppose that $\mathfrak{d} \in \mathcal{D}$, $\omega \in \Omega_{\delta}$, and $v \in \mathcal{P}_1 \cup \mathcal{P}_2$ such that $v \nmid \mathfrak{d}$. Let $\omega' \in \eta_{\mathfrak{d},v}^{-1}(\omega) \subset \Omega_{\mathfrak{d}v}$. 
    \begin{enumerate}
        \item $|\eta_{\mathfrak{d},v}^{-1}(\omega)| = 6$
        \item Suppose that $v \in \mathcal{P}_1$. Then $0 \leq t(v) \leq 1$ and
        \begin{equation}
            \emph{\text{rk}}(\omega') = \begin{cases}
            \emph{\text{rk}}(\omega) - 1 & \text{ if } t_\omega(v) = 1 \\
            \emph{\text{rk}}(\omega) + 1 & \text{ if } t_\omega(v) = 0.
            \end{cases}
        \end{equation}
        \item Suppose that $v \in \mathcal{P}_2$. Then $0 \leq t(v) \leq 2$ and
        \begin{equation}
            \emph{\text{rk}}(\omega') = \begin{cases}
                \emph{\text{rk}}(\omega) - 2 & \text{ if } t_\omega(v) = 2 \\
                \emph{\text{rk}}(\omega) & \text{ if } t_\omega(v) = 1 \\
                \emph{\text{rk}}(\omega) + 2 & \text{ if } t_\omega(v) = 0  \text{ for exactly 2 of } \omega' \in \eta_{\mathfrak{d},v}^{-1}(\omega) \\
                \emph{\text{rk}}(\omega) & \text{ if } t_\omega(v) = 0  \text{ for the other 4 of } \omega' \in \eta_{\mathfrak{d},v}^{-1}(\omega).
            \end{cases}
        \end{equation}
    \end{enumerate}
\end{proposition}

The above proposition demonstrates that differences in dimensions of local Selmer structures can be obtained from understanding differences in local conditions as $t(v)$ ranges over the set of places $v$ where the local conditions differ. The probability that $t(v)$ is equal to $0, 1, 2$ can be obtained as a Chebotarev density theoretic statement using the following two propositions.
\begin{proposition}[Proposition 9.3 of \cite{KMR14}]\label{prop:KMR9.3}
Suppose $E$ is an elliptic curve over $k$ such that $\emph{\text{Gal}}(k(E[3])/k) \supset \emph{\text{SL}}_2(\mathbb{F}_3)$. Denote by $k_{\mathfrak{d},\omega}$ be the Galois extension of $k$ defined as
\begin{equation}\label{eq:orig_Selmerfieldk}
    k_{\mathfrak{d},\omega} = \emph{\text{Fixed field of }} \bigcap_{c \in \Sel(E[3], \omega)_k} \emph{\text{Ker}}(c: \emph{\text{Gal}}(\overline{k}/k(E[3])) \to E[3]).
\end{equation}
where we identify elements of $\Sel(E[3],\omega)_k$ to lie in $H^1(k(E[3]), E[3])$ via the inflation-restriction sequence of Galois cohomology groups. Then for every $\mathfrak{d} \in \mathcal{D}$ and $\omega \in \Omega_{\mathfrak{d}}$:
\begin{enumerate}
    \item There exists a $\emph{\text{Gal}}(k(E[3])/k)$ module isomorphism $\emph{\text{Gal}}(k_{\mathfrak{d},\omega}/k(E[3])) \cong E[3]^{\emph{\text{rk}}(\omega)}$.
    \item The map $\Sel(E[3],\omega)_k \to \emph{\text{Hom}}(\emph{\text{Gal}}(\overline{k}/k(E[3])), E[3])$ induces isomorphisms
    \begin{align}
        \begin{split}
            \Sel(E[3], \omega)_k &\cong \emph{\text{Hom}}(\emph{\text{Gal}}(k_{\mathfrak{d},\omega}/k(E[3])), E[3])^{\emph{\text{Gal}}(k(E[3])/k)}, \\
            \emph{\text{Gal}}(k_{\mathfrak{d},\omega}/k(E[3])) &\cong \emph{\text{Hom}}(\Sel(E[3],\omega)_k, E[3]).
        \end{split}
    \end{align}
    \item The extension $k_{\mathfrak{d},\omega}/k$ is unramified away from places dividing $\Sigma(\mathfrak{d})$.
\end{enumerate}
\end{proposition}

Before we can state some result on Markov operators governing variations of Selmer ranks, we need to first give an effective Chebotarev density theorem from \cite{KMR14} which can be used in giving rates of convergence for Markov operators given later in this section and also understanding the frequency with which ``random'' Selmer structures have $\mathbb{F}_3$-dimensions of various size after localization. 

\begin{theorem}[An Effective Chebotarev Density Theorem, Theorem 8.1 of \cite{KMR14}]\label{thrm:KMR8.1:effectivecheb}

    There is a nondecreasing function $\mathcal{L}:[1, \infty) \to [1, \infty)$ such that for 
    \begin{itemize}
        \item every $Y\geq 1$, 
        \item every $\mathfrak{d} \in \mathcal{D}$ with $\textbf{N}\mathfrak{d} < Y$, 
        \item every Galois extension $L/k$ that is abelian of exponent $p$ over $k(E[3])$ and unramified outside of $\Sigma(\mathfrak{d})$, 
        \item every pair of subset $S,S' \subset \mathrm{Gal}(L/K)$ stable under conjugation with $S$ non-empty, and 
        \item every $X > \mathcal{L}(Y)$,
        \end{itemize}
        we have 
        \begin{equation}
            \left| 
            \frac{
            \left|\left\{
            \mathfrak{q} \notin \Sigma(\mathfrak{d}) : \textbf{N}\mathfrak{q} \leq X, \mathrm{Frob}_{\mathfrak{q}}(L/k) \in S'
            \right\} \right|
            }{
            \left|\left\{ \mathfrak{q} \notin \Sigma(\mathfrak{d}) : \textbf{N}\mathfrak{q} \leq X, \mathrm{Frob}_{\mathfrak{q}}(L/k) \in S\right\} \right|
            } 
            - \frac{|S'|}{|S|}\right| \leq \frac{1}{Y}.
        \end{equation}
        Further, $\{\mathfrak{q} \notin \Sigma(\mathfrak{d}) : \textbf{N}\mathfrak{q} \leq X, \mathrm{Frob}_{\mathfrak{q}}(L/k) \in S\}$ is nonempty. 
\end{theorem}

\begin{proposition}[Proposition 9.4 of \cite{KMR14}]
    Fix $\mathfrak{d} \in \mathcal{D}$ and $\omega \in \Omega_{\mathfrak{d}}$. For every $\mathfrak{q} \notin \Sigma(\mathfrak{d})$, define 
    \begin{equation}\label{eq:def:Tq}
        t_\omega(\mathfrak{q}) = \dim_{\mathbb{F}_3}\left( \mathrm{loc}_{\mathfrak{q}} \mathrm{Sel}(E[3], \omega)_k\right). 
    \end{equation}
    Let $c_{i,j}$ be given by the entries in the following table: 
    \begin{center}
        \begin{tabular}{c||c|c|c}
             & $j=0$ & $j=1$ & $j=2$\\ \hline \hline
            $i=1$ & $3^{-\mathrm{rk}(\omega)}$ & $1-3^{-\mathrm{rk}(\omega)}$& 0 \\
            $i=2$ & $3^{-2\mathrm{rk}(\omega)}$ & $4(3^{-\mathrm{rk}(\omega)} - 3^{-2\mathrm{rk}(\omega)})$ &  $1-4 \cdot 3^{-\mathrm{rk}(\omega)} + 3^{1-2\mathrm{rk}(\omega)}$
        \end{tabular}
    \end{center}
        Then for $i=2$ and $j=0,1,2$, 
        \begin{equation}\label{eqn:9.4prob}
        \lim_{X \to \infty} \frac{\#\{\mathfrak{q} \in \mathcal{P}_i(X) \mid \mathfrak{q} \nmid \mathfrak{d}, t(\mathfrak{q})=j\}}{\#\{\mathfrak{q} \in \mathcal{P}_i(X) \mid \mathfrak{q} \nmid \mathfrak{d}\}} = c_{i,j};
        \end{equation}
        further, if $\mathcal{L}$ is a function satisfying Theorem \ref{thrm:KMR8.1:effectivecheb}, then for every $Y > {\bf N}\mathfrak{d}$ and every $X > \mathcal{L}(Y)$ we have 
        \begin{equation}\label{eqn:9.4convergence}
            \left| \frac{\#\{\mathfrak{q} \in \mathcal{P}_i(X) \mid \mathfrak{q} \nmid \mathfrak{d}, t(\mathfrak{q})=j\}}{\#\{\mathfrak{q} \in \mathcal{P}_i(X) \mid \mathfrak{q} \nmid \mathfrak{d}\}} - c_{i,j} \right| \leq \frac{1}{Y}.
        \end{equation}
        Finally, if it also the case $3 \mid [K(E[3]) : K]$, then \eqref{eqn:9.4prob} and \eqref{eqn:9.4convergence} are also true for $i=1$ and $j=0,1,2$.
\end{proposition}

\begin{proposition}[Proposition 9.5 of \cite{KMR14}]
For every subset $S \subset \Omega_1$, the $\mod 3$ Lagrangian Markov operator $M_L = [m_{r,s}]$, where 
$$m_{r,s} = \begin{cases}
    1 - 3^{-r} \quad \text{if}~ s=r-1 \geq 0 \\
    3^{-r} \quad \text{if}~ s=r+1 \geq 1 \\
    0  \quad \text{otherwise}, 
\end{cases}$$
on the space 
$$\left\{
\text{maps}~ f: \mathbb{Z}_{\geq 0} \to \mathbb{R} ~\Big|~ ||f||:= \sum_{n \in \mathbb{Z}_{\geq 0}} |f(n)| ~ \text{converges} 
\right\}$$
governs the rank data $\Omega^S$ on $\mathcal{D}$ in that 
\begin{equation}
m^{(i)}_{\mathrm{rk}(\omega),s} = \frac{\sum_{q \in \mathcal{P}_i(X)-\delta}\left| \left\{ \chi \in \eta^{-1}_{\delta,q}(\omega) : \mathrm{rk}(\chi)=s\right\}\right|}{\sum_{q \in \mathcal{P}_i(X)-\delta}\left| \left\{ \eta^{-1}_{\delta,q}(\omega)\right\}\right|} 
\end{equation}
where $M^i= [m_{r,s}^{(i)}]$.

Moreover, every function $\mathcal{L}$ satisfying Theorem \ref{thrm:KMR8.1:effectivecheb} is a convergence rate for $(\Omega^S, M_L)$.
\end{proposition} 

\subsection{Controlling Localizations via Corestrictions}
For a number field $L$ and an elliptic curve $E/L$, write $\mathcal{P}_{j}^L$ be the primes $p$ of $L$ for which $\dim_{\mathbb{F}_3}H^1_{ur}(L_p, E[3]) = j$. We first prove the following results that hold for any set of square-free elements $\mathcal{D}$ and set of primes $\mathcal{P}_j^k$. As such, we do not yet make the choice of $\mathcal{D}^*$ and the according set of primes $\mathcal{P}_j^{k,*}$. 
\begin{lemma}\label{lem:corestrictions}
    Suppose that $\mathfrak{d} \in \mathcal{D}$ and $\omega \in \Omega_{\mathfrak{d}}$.
    Let $p \in \mathcal{P}_{j}^k$, $j=0,1,2$, be a prime of $k$ that splits completely in $F$ as $p\mathcal{O}_F = \mathfrak{p}_0\mathfrak{p}_1$. Then for any $c \in \mathrm{Sel}(E[3],\omega)_F$, whose image under the localization map $\mathrm{loc}_p$ is identified in $H^1_{ur}(k_p, E[3]) \simeq \mathbb{F}_3$ or $\mathbb{F}_3^2$,  we have 
    \begin{equation}\label{eq:localizationcores}
        \mathrm{loc}_{\mathfrak{p}_0}(c) + \mathrm{loc}_{\mathfrak{p}_1}(c) \equiv \mathrm{loc}_p \mathrm{cor}_{F/k}(c) \pmod{3},
    \end{equation}
    as vectors in $\mathbb{F}_3^j$.
\end{lemma}

\begin{proof}
Since $p$ splits completely in $F$, we have $k_p = F_{\mathfrak{p}_0}=F_{\mathfrak{p}_1}$, and so 
\[H^1_{ur}(k_p, E[3]) = H^1_{ur}(F_{\mathfrak{p}_0}, E[3]) = H^1_{ur}(F_{\mathfrak{p}_1}, E[3]).\]
Recall the three local cohomology groups above are the images of the appropriate Selmer group in $H^1(k_p, E[3])$. For a prime $p \in \mathcal{P}_{j}^F$ with $j=1$ or $j=2$, we have the following commutative diagram.

\begin{tikzcd}
\mathrm{Sel}(E[3],\omega)_F \arrow[d, "\mathrm{cor}_{F/k}"'] \arrow[rr, "\mathrm{loc}_{\mathfrak{p_1}} \oplus \mathrm{loc}_{\mathfrak{p}_1}"] &  & {H^1_{ur}(F_{\mathfrak{p}_0}, E[2]) \oplus H^1_{ur}(F_{\mathfrak{p}_1}, E[3])} \arrow[r, "\simeq"] \arrow[d, "\mathrm{cor}_{F/k}c_{\mathfrak{p}_0} + \mathrm{cor}_{F/k}c_{\mathfrak{p}_1}"] & {H_{ur}^1(k_p, E[3])^{\oplus 2}} \arrow[r, "\simeq"] & \mathbb{F}_3^{2j} \arrow[ld, "v_0 + v_1"] \\
\mathrm{cor}_{F/k} \left(\mathrm{Sel}(E[3],\omega)_F\right) \arrow[rr, "\mathrm{loc}_p"]                                                                                        &  & {H_{ur}^1(k_p, E[3])} \arrow[r, "\simeq"]                                                                                                                                            & \mathbb{F}_3^j                                      &                                       
\end{tikzcd}

Note that the leftmost corestriction map $\mathrm{cores}_{F/k}$ is an isomorphism. This is because one has $\mathrm{cores}_{F/k} \circ \mathrm{res}_{F/k}: H^1(F,E[3]) \to H^1(F,E[3])$ is a multiplication by $2$ map, and $\mathrm{Sel}(E[3],\omega)_F$ is an $\mathbb{F}_3$-vector space.
\end{proof}

\begin{remark}
    The utility of \eqref{eq:localizationcores} is the following: we will use the methods of \cite{KMR14} to understand the distribution of primes $p \in \mathcal{P}_{j}^k$ of $k$ for which $\mathrm{loc}_p \mathrm{cor}_{F/k}(c)$ is some fixed element of $\mathbb{F}_3^k$. From this, we can understand the possibilities for $ \mathrm{loc}_{\mathfrak{p}_0}(c)$ and $ \mathrm{loc}_{\mathfrak{p}_1}(c)$, which essentially determine the size of the Selmer group of $E$ upon base change of an $S_3$-cubic extension with quadratic resolvent field $F$. 
\end{remark}

Recall from \eqref{eq:orig_Selmerfieldk} that $k_{\mathfrak{d},\omega}$ is the abelian extension of $k(E[3])$ obtained by taking the fixed field of the intersection of kernels of the restrictions to $k(E[3])$ of each element of $\Sel(E, \omega)_k$. As in \eqref{eq:orig_Selmerfieldk}, for a quadratic extension $F/k$ we define 
\begin{equation}\label{eq:SelmerfieldF}
    F_{\mathfrak{d},\omega} = \text{Fixed field of } \bigcap_{c \in \Sel(E[3], \omega)_F} \text{Ker}(c: \text{Gal}(\overline{k}/F(E[3])) \to E[3]).
\end{equation}
We will use the notation $k_{\mathfrak{d},\omega}$ to denote instead
\begin{equation}\label{eq:Selmerfieldk}
    k_{\mathfrak{d},\omega} = \text{Fixed field of } \bigcap_{c \in \mathrm{cor}_{F/k}(\Sel(E[3], \omega)_F)} \text{Ker}(c: \text{Gal}(\overline{k}/k(E[3])) \to E[3]).
\end{equation}
By Proposition \ref{prop:KMR9.3}, we have that 
\[\mathrm{Gal}(k_{\mathfrak{d},\omega}/k(E[3])) \simeq \mathrm{Hom}(\mathrm{cor}_{F/k}(\Sel(E, \omega)_F), E[3]) \simeq E[3]^{\dim_{\mathbb{F}_3} \mathrm{Sel}(E, \omega)_F} \]
and 
\[\mathrm{Gal}(F_{\mathfrak{d},\omega}/F(E[3])) \simeq \mathrm{Hom}(\Sel(E, \omega)_F, E[3]) \simeq E[3]^{\dim_{\mathbb{F}_3} \mathrm{Sel}(E, \omega)_F}. \] 

\begin{lemma}
\begin{enumerate}
    \item[] With the notation as above:
    \item We have $F_{\mathfrak{d},\omega} = F k_{\mathfrak{d},\omega}$.
    \item We have
    \begin{align}
        \begin{split}
            \emph{\text{Gal}}(k_{\mathfrak{d},\omega}/k) &\cong \emph{\text{SL}}_2(\mathbb{F}_3) \rtimes E[3]^{\dim_{\mathbb{F}_3} \mathrm{Sel}(E, \omega)_F}, \quad  \text{and} \\
            \emph{\text{Gal}}(F_{\mathfrak{d},\omega}/k) &\cong \mathbb{Z}/2\mathbb{Z} \oplus \left( \emph{\text{SL}}_2(\mathbb{F}_3) \rtimes E[3]^{\dim_{\mathbb{F}_3} \mathrm{Sel}(E, \omega)_F} \right).
        \end{split}
    \end{align}
\end{enumerate}
\end{lemma}
\begin{proof}
    Because the group $\text{SL}_2(\mathbb{F}_3)$ does not have any index $2$ normal subgroup, we have $F \cap k(E[3]) = k$. The isomorphism $\mathrm{cores}_{F/k}: \mathrm{Sel}(E,\omega)_F \to \mathrm{cor}_{F/k} (\mathrm{Sel}(E,\omega)_k)$ implies that $\dim_{\mathbb{F}_3} \mathrm{Sel}(E, \omega)_F = \dim_{\mathbb{F}_3} \mathrm{cor}_{F/k} (\mathrm{Sel}(E, \omega)_F)$, from which all the statements of the lemma follow.
\end{proof}

For the rest of this section, let $q$ denote a prime of $k$, let $\mathfrak{q}$ denote a prime above $q$ in $k_{\mathfrak{d},\omega}$, and let $\mathfrak{Q}$ denote a prime above $q$ in $F_{\mathfrak{d},\omega}$.

For now, we only consider primes $q$ of $K$ which split into two primes in $F$, say $q\mathcal{O}_F = Q_0Q_1$. Consequently, because $F$ and $k_{\mathfrak{d}, \omega}$ are linearly disjoint, $\mathfrak{q}$ is forced to split completely as $\mathfrak{q}\mathcal{O}_{F_{\mathfrak{d}, \omega}}=\mathfrak{Q}_0\mathfrak{Q}_1$. The Artin symbol
$$\left( \frac{k_{\mathfrak{d},\omega}/k}{\mathfrak{q}} \right)$$
determines localization at $q$. The two Artin symbols
$$\left( \frac{F_{\mathfrak{d},\omega}/k}{\mathfrak{Q}_i} \right)$$
determine localization at $\mathfrak{q}_i$.
As elements, for $i=0,1$, 
$$\left( \frac{F_{\mathfrak{d},\omega}/k}{\mathfrak{Q}_i} \right)^{f(Q_i \mid q)}  = \left( \frac{F_{\mathfrak{d},\omega}/k}{\mathfrak{Q}_i} \right)^{1} = \left( \frac{F_{\mathfrak{d},\omega}/F}{\mathfrak{Q}_i} \right).$$
Hence, there are 3 Artin symbols in play, all of which lie in a common conjugacy class of $\text{Gal}(k_{\mathfrak{d},\omega}/k)$:
\begin{equation*}
    \left( \frac{k_{\mathfrak{d},\omega}/k}{\mathfrak{q}} \right), \left( \frac{F_{\mathfrak{d},\omega}/k}{\mathfrak{Q}_0} \right), \left( \frac{F_{\mathfrak{d},\omega}/k}{\mathfrak{Q}_1} \right).
\end{equation*}

Because $\mathfrak{Q}_0, \mathfrak{Q}_1$ are primes lying above $\mathfrak{q}$, we use the canonical restriction map which forgets the $\mathbb{Z}/2\mathbb{Z}$ component
\begin{equation*}
\text{res}_{F_{\mathfrak{d},\omega}/k_{\mathfrak{d},\omega}}: \text{Gal}(F_{\mathfrak{d},\omega}/k) \to \text{Gal}(k_{\mathfrak{d},\omega}/k)
\end{equation*}
to obtain
\begin{equation*}
\text{res}_{F_{\mathfrak{d},\omega}/k_{\mathfrak{d},\omega}} \left( \frac{F_{\mathfrak{d},\omega}/k}{\mathfrak{Q}_0} \right) = \left( \frac{k_{\mathfrak{d},\omega}/k}{\mathfrak{Q}_0 \cap k_{\mathfrak{d},\omega}} \right) = \left( \frac{k_{\mathfrak{d},\omega}/k}{\mathfrak{q}} \right).
\end{equation*}
The identical equation can be obtained for $\mathfrak{Q}_1$ as well. Because $\text{res}_{F_{\mathfrak{d},\omega}/k_{\mathfrak{d},\omega}}$ induces an isomorphism between $\text{Gal}(F_{\mathfrak{d},\omega}/F)$ and $\text{Gal}(k_{\mathfrak{d},\omega}/k)$, we in fact obtain
\begin{equation*}
    \left( \frac{F_{\mathfrak{d},\omega}/k}{\mathfrak{Q}_0} \right) = \left( \frac{k_{\mathfrak{d},\omega}/k}{\mathfrak{q}} \right) = \left( \frac{F_{\mathfrak{d},\omega}/k}{\mathfrak{Q}_1} \right).
\end{equation*}

Because we have that the 3 Artin symbols are equal to one another, Lemma \ref{lem:corestrictions} implies that for each prime $\mathfrak{q}$ lying above $p$,
\begin{align*}
    \text{loc}_p(\text{cores}_{F/k}(c)) \left( \frac{k_{\mathfrak{d},\omega}/k}{\mathfrak{q}} \right) &= \text{loc}_p(\text{cores}_{F/k}(c)) (\text{Frob}_p) \\
    &= \text{loc}_{\mathfrak{p}_0}(c) (\text{Frob}_{\mathfrak{p}_0}) + \text{loc}_{\mathfrak{p}_1}(c) (\text{Frob}_{\mathfrak{p}_1}) \\
    &= \text{loc}_{\mathfrak{p}_0}(c) \left( \frac{F_{\mathfrak{d},\omega}/k}{\mathfrak{Q}_0} \right) + \text{loc}_{\mathfrak{p}_1}(c) \left( \frac{F_{\mathfrak{d},\omega}/k}{\mathfrak{Q}_1} \right). 
\end{align*}
Because the 3 Artin symbols are equal to each other and the localization map at $\mathfrak{p}_i$ is defined as evaluating $c$ at the Frobenius element of $\mathfrak{p}_i$, we have
\begin{equation*}
    \text{loc}_{\mathfrak{p}_0}(c) \left( \frac{F_{\mathfrak{d},\omega}/k}{\mathfrak{Q}_0} \right) = \text{loc}_{\mathfrak{p}_1}(c) \left( \frac{F_{\mathfrak{d},\omega}/k}{\mathfrak{Q}_1} \right).
\end{equation*}
This implies
\begin{align*}
    \text{loc}_p(\text{cores}_{F/k}(c)) \left( \frac{k_{\mathfrak{d},\omega}/k}{\mathfrak{q}} \right) &= 2 \text{loc}_{\mathfrak{p}_0}(c) \left( \frac{F_{\mathfrak{d},\omega}/k}{\mathfrak{Q}_0} \right), \\
    2 \text{loc}_p(\text{cores}_{F/k}(c)) \left( \frac{k_{\mathfrak{d},\omega}/k}{\mathfrak{q}} \right) &= \text{loc}_{\mathfrak{p}_0}(c) \left( \frac{F_{\mathfrak{d},\omega}/k}{\mathfrak{Q}_0} \right) = \text{loc}_{\mathfrak{p}_1}(c) \left( \frac{F_{\mathfrak{d},\omega}/k}{\mathfrak{Q}_1} \right).
\end{align*}
This relation can also be obtained from the fact that the composition of the corestriction map $\text{cores}_{F/k}$ with the restriction map $\text{res}_{F/k}$ is a  multiplication by $2$ over the Selmer group $\text{Sel}(E[3],\omega)_F$.

\begin{definition}
Let $\mathfrak{d} \in \mathcal{D}$ and $\omega \in \Omega_\mathfrak{d}$. Given a prime $p$ over $k$, we denote by
    \begin{align}
        t_\omega(p) &:= \dim_{\mathbb{F}_3} \text{im} \left(\text{loc}_{p}: \text{cor}_{F/k}(\Sel(E[3],\omega)_F) \to H^1_{ur}(k_p,E[3]) \right).
    \end{align}
\end{definition}

With these notations in hand, we can analyze how changes in local characters at a single prime $p$ that splits over $F/k$ affect the dimensions of local Selmer groups.
\begin{lemma} \label{lem:split_prime_change}
    Let $\mathfrak{d} \in \mathcal{D}$ and $\omega \in \Omega_\mathfrak{d}$. Let $p \in \mathcal{P}^k_j$, $j = 0,1,2$ be a prime of $k$ that splits completely in $F$ as $p \mathcal{O}_F = \mathfrak{p}_0 \mathfrak{p}_1$. Then for any $\omega' \in \eta^{-1}_{\mathfrak{d},p}(\omega) \in \Omega_{\mathfrak{d}p}$, 
    \begin{align}
    \begin{split}
        & \; \; \; \; \dim_{\mathbb{F}_3} \Sel(E[3],\omega')_F - \dim_{\mathbb{F}_3} \Sel(E[3],\omega)_F = \begin{cases}
            0 &\text{ if } p \in \mathcal{P}^k_0 \\
            2 &\text{ if } p \in \mathcal{P}^k_1 \text{ and } t_\omega(p) = 0 \\
            -2 &\text{ if } p \in \mathcal{P}^k_1 \text{ and } t_\omega(p) = 1 \\
            4 &\text{ if } p \in \mathcal{P}^k_2 \text{ and } t_\omega(p) = 0 \\
            & \; \; \; \; \text{ for exactly } 2 \text{ of } \omega' \in \eta^{-1}_{\mathfrak{d},p}(\omega)\\
            0 &\text{ if } p \in \mathcal{P}^k_2 \text{ and } t_\omega(p) = 0 \\ & \; \; \; \; \text{ for the other } 4 \text{ of } \omega' \in \eta^{-1}_{\mathfrak{d},p}(\omega) \\
            0 &\text{ if } p \in \mathcal{P}^k_2 \text{ and } t_\omega(p) = 1 \\
            -4 &\text{ if } p \in \mathcal{P}^k_2 \text{ and } t_\omega(p) = 2.
        \end{cases}
    \end{split}
    \end{align}
\end{lemma}
\begin{proof}
    Given an ideal $I \subset \mathcal{O}_k$, denote by
    \begin{align}
    \begin{split}
        \Sel(E[3],\omega)_{k}^I &:= \text{Ker} \left( \oplus_{v \nmid I} \text{loc}_v: H^1(k,E[3]) \to \oplus_{v \nmid I} H^1(k_p,E[3]) / \mathcal{H}_{\omega_v} \right), \\
        \Sel(E[3],\omega)_{k,I} &:= \text{Ker} \left( \oplus_{v \mid I} \text{loc}_v: \Sel(E[3],\omega)_k^I \to \oplus_{v \mid I} H^1(k_v, E[3]) \right).
    \end{split}
    \end{align}
    As stated in Proposition 7.2 of \cite{KMR14} and Theorem 3.9 of \cite{KMR13}, we have the following two short exact sequences, where we let $I = p \mathcal{O}_F = \mathfrak{p}_0 \mathfrak{p}_1$. For short-hand notation, we will use $(p)$ to denote $p \mathcal{O}_F$.
    \begin{align} \label{eqn:ses_Selmer_comparison}
        \begin{split}
            0 \to \Sel(E[3],\omega)_{F,(p)} \to &\Sel(E[3],\omega)_F \to (\oplus_{i=0}^1 \text{loc}_{\mathfrak{p}_i})(\Sel(E[3],\omega)_F^{(p)}) \cap (\oplus_{i=0}^1 \mathcal{H}_{\omega_{\mathfrak{p}_i}}) \to 0. \\
            0 \to \Sel(E[3], \omega)_{F,(p)} \to &\Sel(E[3],\omega')_F \to (\oplus_{i=0}^1 \text{loc}_{\mathfrak{p}_i})(\Sel(E[3],\omega)_F^{(p)}) \cap (\oplus_{i=0}^1 \mathcal{H}_{\omega'_{\mathfrak{p}_i}}) \to 0.
        \end{split}
    \end{align}
    We note that for all $i = 0,1$,
    \begin{align*}
        \mathcal{H}_{\omega_{\mathfrak{p}_0}} = \mathcal{H}_{\omega_{\mathfrak{p}_1}} &= H^1_{ur}(k_p, E[3]), \\
        \mathcal{H}_{\omega'_{\mathfrak{p}_0}} = \mathcal{H}_{\omega'_{\mathfrak{p}_1}} &\neq H^1_{ur}(k_p, E[3]),
    \end{align*}
    because these spaces are images of local Kummer maps, both of which are isomorphic to $k_p$ rational points of the 2-dimensional abelian variety $A_{K/k}$.
    Because the space
    \begin{equation*}
        (\oplus_{i=0}^1 \text{loc}_{\mathfrak{p}_i})(\Sel(E[3],\omega)_F^{(p)}) \subset \oplus_{i=0}^1 H^1(k_p,E[3])
    \end{equation*}
    is a Lagrangian subspace with respect to the sum of local Tate pairings (as shown in Theorem 3.9 of \cite{KMR13}), we can compare the dimensions to obtain the isomorphism
    \begin{align}
        (\oplus_{i=0}^1 \text{loc}_{\mathfrak{p}_i})(\Sel(E[3],\omega)_F^{(p)}) &\cong \text{loc}_{\mathfrak{p}_0}(\Sel(E[3],\omega)_F^{\mathfrak{p}_0}) \oplus \text{loc}_{\mathfrak{p}_1}(\Sel(E[3],\omega)_F^{\mathfrak{p}_1}).
    \end{align}
    Denote by $F_{\mathfrak{d},\omega}^{(p)}$ the Galois extension over $F$ defined as
    \begin{equation}
        F_{\mathfrak{d},\omega}^{(p)} := \text{Fixed field of } \bigcap_{c \in \Sel(E[3],\omega)_F^{(p)}} \text{Ker} \left(c: \text{Gal}(\overline{k}/F(E[3])) \to E[3]\right). 
    \end{equation}
    Following the proof of Lemma \ref{lem:corestrictions}, the images of the two localization maps $\text{loc}_{\mathfrak{p}_0}$ and $\text{loc}_{\mathfrak{p}_1}$ of $\Sel(E[3],\omega)_F^{(p)}$ are identical. Hence, we have for any Lagrangian subspaces $\mathcal{H}_{\omega_{\mathfrak{p}_i}}$ and $\mathcal{H}_{\omega'_{\mathfrak{p}_i}}$,
    \begin{align}
        \dim_{\mathbb{F}_3} (\oplus_{i=0}^1 \text{loc}_{\mathfrak{p}_i})(\Sel(E[3],\omega)_F^{(p)}) \cap (\oplus_{i=0}^1 \mathcal{H}_{\omega_{\mathfrak{p}_i}}) &\equiv 0 \text{ mod } 2, \\
        \dim_{\mathbb{F}_3} (\oplus_{i=0}^1 \text{loc}_{\mathfrak{p}_i})(\Sel(E[3],\omega)_F^{(p)}) \cap (\oplus_{i=0}^1 \mathcal{H}_{\omega'_{\mathfrak{p}_i}}) &\equiv 0 \text{ mod } 2.
    \end{align}
    Recall that 
    \begin{align*}
        t_\omega(p) &:= \dim_{\mathbb{F}_3} \text{im} \left(\text{loc}_{p}: \text{cor}_{F/k}(\Sel(E[3],\omega)_F) \to H^1_{ur}(k_p,E[3]) \right).
    \end{align*}
    Using the isomorphism induced from the corestriction map $\text{cor}_{F/k}$, we can apply the proof of Proposition 7.2 of \cite{KMR14} to obtain
    \begin{align}
        \dim_{\mathbb{F}_3} (\oplus_{i=0}^1 \text{loc}_{\mathfrak{p}_i})(\Sel(E[3],\omega)_F^{(p)}) \cap (\oplus_{i=0}^1 \mathcal{H}_{\omega_{\mathfrak{p}_i}}) &= 2 \cdot t_\omega(p).
    \end{align}
    We hence obtain
    \begin{align}
    \begin{split}
        &\; \; \; \; \dim_{\mathbb{F}_3} \Sel(E[3],\omega')_F - \dim_{\mathbb{F}_3} \Sel(E[3],\omega)_F \\
        &= \dim_{\mathbb{F}_3} (\oplus_{i=0}^1 \text{loc}_{\mathfrak{p}_i})(\Sel(E[3],\omega)_F^{(p)}) \cap (\oplus_{i=0}^1 \mathcal{H}_{\omega'_{\mathfrak{p}_i}}) - 2 \cdot t_\omega(p).
    \end{split}
    \end{align}
    One thing to notice here is that if $\dim_{\mathbb{F}_3} \Sel(E[3], \omega)_F \leq 1$, then $t_\omega(p) = 0$. Indeed, this follows from the surjection stated in (\ref{eqn:ses_Selmer_comparison}). Likewise, if $\dim_{\mathbb{F}_3} \Sel(E[3], \omega)_F \leq 3$, then $t_\omega(p) \neq 2$.
    
    We proceed to prove the rest of the lemma by considering two cases where $p \in \mathcal{P}_1^k$ or $p \in \mathcal{P}_2^k$. 

    \medskip 
    
    \textbf{Case 1: $j = 1$}
    
    \medskip 
    
    Suppose $p \in \mathcal{P}_1^k$. Then because $H^1_{ur}(k_p,E[3])$ has a unique unramified Lagrangian subspace and a unique ramified Lagrangian subspace,
    \begin{equation}
        \dim_{\mathbb{F}_3} (\oplus_{i=0}^1 \text{loc}_{\mathfrak{p}_i})(\Sel(E[3],\omega)_F^{(p)}) \cap (\oplus_{i=0}^1 \mathcal{H}_{\omega'_{\mathfrak{p}_i}}) = 2 - 2 \cdot t_\omega(p).
    \end{equation}
    Hence, we obtain
    \begin{equation}
        \dim_{\mathbb{F}_3} \Sel(E[3],\omega')_F - \dim_{\mathbb{F}_3} \Sel(E[3],\omega)_F = \begin{cases}
            2 &\text{ if } t_\omega(p) = 0 \\
            -2 &\text{ if } t_\omega(p) = 1.
        \end{cases}
    \end{equation}

    \medskip 
    
    \textbf{Case 2: $j = 2$}

    Suppose $p \in \mathcal{P}_2^k$. Then because $H^1_{ur}(k_p,E[3])$ has a unique unramified Lagrangian subspace and $3$ ramified Lagrangian subspaces,
    \begin{equation}
        \dim_{\mathbb{F}_3} (\oplus_{i=0}^1 \text{loc}_{\mathfrak{p}_i})(\Sel(E[3],\omega)_F^{(p)}) \cap (\oplus_{i=0}^1 \mathcal{H}_{\omega'_{\mathfrak{p}_i}}) = \begin{cases}
            4 &\text{ if } t_\omega(p) = 0 \text{ for one of the Lag. sbsp. } \\
            2 &\text{ if } t_\omega(p) = 0 \text{ for two other Lag. sbsp. } \\
            2 &\text{ if } t_\omega(p) = 1 \\
            0 &\text{ if } t_\omega(p) = 2.
        \end{cases}
    \end{equation}
    Hence, we obtain
    \begin{equation}
        \dim_{\mathbb{F}_3} \Sel(E[3],\omega')_F - \dim_{\mathbb{F}_3} \Sel(E[3],\omega)_F = \begin{cases}
            4 &\text{ if } t_\omega(p) = 0 \text{ for one of the Lag. sbsp. }\\
            0 &\text{ if } t_\omega(p) = 0 \text{ for two other Lag. sbsp. } \\
            0 &\text{ if } t_\omega(p) = 1 \\
            -4 &\text{ if } t_\omega(p) = 2.
        \end{cases}
    \end{equation}
\end{proof}

Lemma \ref{lem:split_prime_change} shows that the set of all possible images of the localization maps $\text{loc}_{\mathfrak{p}_0} \oplus \text{loc}_{\mathfrak{p}_1}$ is identical to the set of all possible subsets of $H^1(k_v, \text{Res}_k^F E[3])$ that is equal to $H^1_f(k_v, B_{K/k}[1-\sigma_K])$ appearing in Remark \ref{remark:local-conditions-Kummer}. This leverages us to use the corestriction map to govern distribution of images of local Kummer maps by using local characters in $\Omega_\mathfrak{d}$ for each square-free $\mathfrak{d}$.

We now turn our focus to analyzing how changes in local characters at an inert prime $p$ over $F = k(\zeta_3)/k$ affect the dimensions of local Selmer groups. 
\begin{lemma} \label{lem:inert_Pis}
    Suppose $\emph{\text{Gal}}(k(E[3])/k) = \emph{\text{GL}}_2(\mathbb{F}_3)$. Let $p$ be a prime of $k$ that is inert over $k(\zeta_3)/k$. Then the following equivalence relation holds.
    \begin{enumerate}
        \item $p \in \mathcal{P}_1^k \cap \mathcal{P}_2^F$ if and only if $\emph{\text{Frob}}_p$ has order $2$ in $\emph{\text{Gal}}(k(E[3])/k)$.
        \item $p \in \mathcal{P}_0^k \cap \mathcal{P}_0^F$ if and only if $\emph{\text{Frob}}_p$ has order $8$ in $\emph{\text{Gal}}(k(E[3])/k)$.
    \end{enumerate}
\end{lemma}
We note that because $\text{GL}_2(\mathbb{F}_3)$ has $12$ elements of order $2$ and $12$ elements of order $8$ with determinant $-1$, Chebotarev density theorem implies that approximately half of inert primes $p$ lie in $\mathcal{P}_1^k \cap \mathcal{P}_2^F$, and the other half lie in $\mathcal{P}_0^k \cap \mathcal{P}_0^F$.
\begin{proof}
    There are $24$ elements in $\text{GL}_2(\mathbb{F}_3)$ whose determinant is equal to $-1$. The $12$ elements of order $2$ form a single conjugacy class of $\text{GL}_2(\mathbb{F}_3)$, and the other $12$ elements of order $8$ form two conjugacy classes, each consisting of $6$ elements. Over the field $F = k(\zeta_3)$, we have $\text{Gal}(k(E[3])/F) = \text{SL}_2(\mathbb{F}_3)$. Apply Proposition 5.9 of \cite{KMR14} to obtain that $p \in \mathcal{P}_2^F$ if and only if $\text{Frob}_p$ has order $2$ in $\text{Gal}(k(E[3])/k)$.

    Because $\zeta_3 \not\in k$, we have $p \not\in \mathcal{P}_2^k$. We now show that $p \in \mathcal{P}_1^k$ if and only if $p \in \mathcal{P}_2^F$. Suppose $p \in \mathcal{P}_1^k$. Then the $3$-division polynomial of $E$, a degree $4$ polynomial whose roots are $x$-coordinates of the $3$-torsion points of $E$, factorizes as
    \begin{equation}
        \psi_3 = (x-a_1)(x-a_2)(x^2 + \alpha x + \beta)
    \end{equation}
    for some $a_1, a_2, \alpha, \beta \in k_p$. Exactly one of $\sqrt{a_i^3 + A a_i + B}$, which is the $y$-coordinate of $3$-torsion point, is in $k_p$. Because $k(\zeta_3)_p/k_p$ is a degree $2$ Galois extension, it follows that $\#E[3](k(\zeta_3)_p) \geq 5$. This implies that $p \in \mathcal{P}_2^F$.

    Suppose on the other hand that $\dim_{\mathbb{F}_3}(k_p) = 0$. Denote by $E^{-3}$ the quadratic twist of $E$ by $-3$. Then it follows that $\dim_{\mathbb{F}_3}E^{-3}[3](k_p) \leq 1$ because $\text{Gal}(k(E^{-3}[3])/k) = \text{GL}_2(\mathbb{F}_3)$ as well. But if $\dim_{\mathbb{F}_3} E^{-3}[3](k_p) = 1$, then we have $\dim_{\mathbb{F}_3}E^{-3}[3](k(\zeta)_p) = 2$. Because
    \begin{equation*}
\dim_{\mathbb{F}_3}E^{-3}[3](k_p) + \dim_{\mathbb{F}_3}E[3](k_p) \geq \dim_{\mathbb{F}_3}E[3](k(\zeta_3)_p),
    \end{equation*}
    we have a contradiction. Hence, $\dim_{\mathbb{F}_3}E^{-3}[3](k_p) = 0$, which implies that $p \in \mathcal{P}_0^F$.
\end{proof}

We can now prove an analogue of Lemma \ref{lem:split_prime_change} for inert primes.
\begin{lemma} \label{lem:inert_prime_change}
    Let $\mathfrak{d} \in \mathcal{D}$ and $\omega \in \Omega_\mathfrak{d}$. Let $p \in \mathcal{P}_j^k$, $j = 0, 1, 2$ be a prime of $k$ that is inert in $F$. Suppose that $F = k(\zeta_3)$. Then for any $\omega' \in \eta_{\mathfrak{d},p}^{-1}(\omega) \in \Omega_{\mathfrak{d}p}$,
    \begin{equation}
        \dim_{\mathbb{F}_3} \Sel(E[3],\omega')_F - \dim_{\mathbb{F}_3} \Sel(E[3],\omega)_F = \begin{cases}
            0 &\text{ if } p \in \mathcal{P}_0^k \\
            2 &\text{ if } p \in \mathcal{P}_1^k \text{ and } t_\omega(p) = 0 \\
            -2 &\text{ if } p \in \mathcal{P}_1^k \text{ and } t_\omega(p) = 1.
        \end{cases}
    \end{equation}
\end{lemma}
\begin{proof}
    The strategy of the proof is similar to the proof of Lemma \ref{lem:split_prime_change}. We have the following two short exact sequences:
    \begin{align} \label{eqn:ses_Selmer_comparison_2}
        \begin{split}
            0 \to \Sel(E[3],\omega)_{F,(p)} \to &\Sel(E[3],\omega)_F \to \text{loc}_p (\Sel(E[3],\omega)_F^{(p)}) \cap \mathcal{H}_{\omega_p} \to 0, \\
            0 \to \Sel(E[3], \omega)_{F,(p)} \to &\Sel(E[3],\omega')_F \to \text{loc}_p(\Sel(E[3],\omega)_F^{(p)}) \cap \mathcal{H}_{\omega'_{p}} \to 0.
        \end{split}
    \end{align}
    By Shapiro's lemma,
    \begin{equation*}
        \mathcal{H}_{\omega_p} = H^1_{ur}(F_p,E[3]) \cong H^1_{ur}(k_p, \text{Ind}_{\text{Gal}(\overline{k_p}/F_p)}^{\text{Gal}(\overline{k_p}/k_p)} E[3]) \cong \text{Ind}_{\{e\}}^{\text{Gal}(F_p/k_p)} (H^1_{ur}(k_p,E[3])).
    \end{equation*}
    Indeed this is the case as shown in Lemma \ref{lem:inert_Pis}, where we have
    \begin{equation*}
        \dim_{\mathbb{F}_3} \mathcal{H}_{\omega_p} = 2 \cdot \dim_{\mathbb{F}_3} H^1_{ur}(k_p, E[3]).
    \end{equation*}
    Likewise, we have
    \begin{equation}
        \mathcal{H}_{\omega'_p} = \text{Ind}_{\{e\}}^{\text{Gal}(F_p/k_p)} \left( \frac{H^1(k_p,E[3])}{H^1_{ur}(k_p,E[3])} \right).
    \end{equation}
    We can apply the proof of Proposition 7.2 of \cite{KMR14} to obtain
    \begin{equation}
        \dim_{\mathbb{F}_3} \text{loc}_p(\Sel(E[3],\omega)_F^{(p)}) \cap \mathcal{H}_{\omega_p} = 2 \cdot t_\omega(p),
    \end{equation}
    and if $p \in \mathcal{P}_1^k$, then
    \begin{equation}
        \dim_{\mathbb{F}_3} \text{loc}_p(\Sel(E[3],\omega)_F^{(p)}) \cap \mathcal{H}_{\omega'_p} = 2 - 2 \cdot t_\omega(p).
    \end{equation}
    Combine both equations to obtain the desired statement of the lemma.
\end{proof}

\begin{remark}
    The significance of Lemma \ref{lem:split_prime_change} and Lemma \ref{lem:inert_prime_change} is that the parity of $\Sel(E[3],\omega)$ gets preserved regardless of consecutively changing the local conditions at primes $p$ of $k$.
\end{remark}

We now obtain Chebotarev statements that determine the values of $t_\omega(p)$ for each case of families of $S_3$-cubics over $k$. Hence, we will make the choice of $\mathcal{D}^*$ and the according set of primes $\mathcal{P}_j^{k,*}$ in accordance to the families of $S_3$-cubics corresponding to cases (A), (B), and (C). An analogue of Proposition 9.4 of \cite{KMR14} can be written as follows.

\begin{proposition}
    Let $\mathcal{D}^* = \mathcal{D}^A, \mathcal{D}^B,$ or $\mathcal{D}^C$. Fix $\mathfrak{d} \in \mathcal{D}^*$ and $\omega \in \Omega_{\mathfrak{d}}$. Denote by $r$ the quantity
    \begin{equation}
        r_\omega := \begin{cases}
            \frac{1}{2} \dim_{\mathbb{F}_3} \Sel(E[3],\omega)_F &\text{ if } \dim_{\mathbb{F}_3} \Sel(E[3],\omega) \equiv 0 \text{ mod } 2\\
            \frac{1}{2} (\dim_{\mathbb{F}_3} \Sel(E[3],\omega)_F-1) &\text{ if } \dim_{\mathbb{F}_3} \Sel(E[3],\omega) \equiv 1 \text{ mod } 2.\\
        \end{cases}
    \end{equation}
    Let $c_{i,j}(r_\omega)$ be given by the entries in the following table: 
    \begin{center}
        \begin{tabular}{c||c|c|c}
             & $j=0$ & $j=1$ & $j=2$\\ \hline \hline
            $i=1$ & $3^{-r_\omega}$ & $1-3^{-r_\omega}$& 0 \\
            $i=2$ & $3^{-2r_\omega}$ & $4(3^{-r_\omega} - 3^{-2r_\omega})$ &  $1-4 \cdot 3^{-r_\omega} + 3^{1-2r_\omega}$
        \end{tabular}
    \end{center}
        Then for $i=2$ and $j=0,1,2$, regardless of whether $p$ splits over $F/k$, or remains inert over $F = k(\zeta_3)/k$,
        \begin{equation}
        \lim_{X \to \infty} \frac{\#\{p \in \mathcal{P}_i^{k,*}(X) \mid p \nmid \mathfrak{d}, t_\omega(p)=j\}}{\#\{p \in \mathcal{P}_i^{k,*}(X) \mid p \nmid \mathfrak{d}\}} = c_{i,j}(r_\omega).
        \end{equation}
        Furthermore, if $\mathcal{L}$ is a function satisfying Theorem 8.1 of \cite{KMR14}, then for every $Y > {\bf N}\mathfrak{d}$ and every $X > \mathcal{L}(Y)$ we have 
        \begin{equation}
            \left| \frac{\#\{p \in \mathcal{P}_i^{k,*}(X) \mid p \nmid \mathfrak{d}, t_\omega(p)=j\}}{\#\{p \in \mathcal{P}_i^{k,*}(X) \mid p \nmid \mathfrak{d}\}} - c_{i,j}(r_\omega) \right| \leq \frac{1}{Y}.
        \end{equation}
        Finally, if it also the case $3 \mid [K(E[3]) : K]$, then the two equations are also true for $i=1$ and $j=0,1,2$.
\end{proposition}

\begin{proof}
    Suppose that $\dim_{\mathbb{F}_3} \Sel(E[3],\omega)_F \equiv 0 \text{ mod } 2$. Denote by $\hat{k}_{\mathfrak{d},\omega}$ the Galois extension over $k$ as follows:
    \begin{equation}
        \hat{k}_{\mathfrak{d},\omega} := \text{Fixed field of } \bigcap_{c \in \text{cor}_{F/k}\Sel(E[3], \omega)_F} \text{Ker}(c: \text{Gal}(\overline{k}/k(E[3])) \to E[3]).
    \end{equation}
    Note that $\hat{k}_{\mathfrak{d},\omega} = k_{\mathfrak{d},\omega}$.
    Apply Proposition 9.4 of \cite{KMR14} with the Galois extension $\hat{k}_{\mathfrak{d},\omega}$ to obtain the desired statement.

    Suppose that $\dim_{\mathbb{F}_3} \Sel(E[3],\omega)_F \equiv 1 \text{ mod } 2$. Then Lemma \ref{lem:split_prime_change} and Lemma \ref{lem:inert_prime_change} show that there exists a distinguished non-trivial element $c^*$ which lies in any $\Sel(E[3],\omega)_{F,(\mathfrak{d})}$ for any $\omega \in \Omega_{\mathfrak{d}}$. Note that $c^*$ maps to a trivial element in $H^1_{ur}(k_p,E[3])$ under the localization map $\text{loc}_\mathfrak{p}$ for any $\mathfrak{p} \mid \mathfrak{d}$. This prompts us to define $\hat{k}_{\mathfrak{d},\omega}$ differently from the previous case as follows:
    \begin{equation}
        \hat{k}_{\mathfrak{d},\omega} := \text{Fixed field of } \bigcap_{\substack{c \in \text{cor}_{F/k}\Sel(E[3], \omega)_F \\ c \not\in \langle c^* \rangle}} \text{Ker}(c: \text{Gal}(\overline{k}/k(E[3])) \to E[3]).
    \end{equation}
    Note that $\hat{k}_{\mathfrak{d},\omega}$ is an index $9$ subfield of $k_{\mathfrak{d},\omega}$.
    Then $\text{Gal}(\hat{k}_{\mathfrak{d},\omega}/k(E[3])) \cong E[3]^{\dim_{\mathbb{F}_3} \Sel(E[3],\omega)_F - 1}$. Apply Proposition 9.4 of \cite{KMR14} with the Galois extension $\hat{k}_{\mathfrak{d},\omega}$ to obtain the desired statement. 
    
    We note that in both cases, the field $\hat{k}_{\mathfrak{d},\omega} \cap FH_3$ is equal to $k$, because all the primes $p \in \mathcal{P}_i^{k,*}$ satisfy the condition that $\text{Frob}_p(FH_3/k) = 1$.
\end{proof}

Applying Proposition 9.5 of \cite{KMR14}, we finally obtain the Markov operator governing the dimensions of $\Sel(E[3],\omega)_F$.
\begin{definition}
    We denote by $M_L = [m_{r,s}]$ the \textbf{alternating} $\mod 3$ Lagrangian Markov operator over the state space $$\left\{
\text{maps}~ f: \mathbb{Z}_{\geq 0} \to \mathbb{R} ~\Big|~ ||f||:= \sum_{s \in \mathbb{Z}_{\geq 0}} |f(s)| ~ \text{converges} 
\right\}$$
defined as
    $$m_{r,s} = \begin{cases}
    1 - 3^{-\frac{r}{2}} \quad \text{if}~ s=r-2 \geq 0, r \equiv 0 \text{ mod } 2 \\
    1 - 3^{-\frac{r-1}{2}} \quad \text{if}~ s=r-2 \geq 0, r \equiv 1 \text{ mod } 2 \\
    3^{-\frac{r}{2}} \quad \text{if}~ s=r+2 \geq 1, r \equiv 0 \text{ mod } 2 \\
    3^{-\frac{r-1}{2}} \quad \text{if}~ s=r+2 \geq 1, r \equiv 1 \text{ mod } 2 \\
    0  \quad \text{otherwise}.
\end{cases}$$
\end{definition}

\begin{proposition}
    Let $\mathcal{D}^* = \mathcal{D}^A, \mathcal{D}^B,$ or $\mathcal{D}^C$. For every subset $S \subset \Omega_1$, the \textbf{alternating} $\mod 3$ Lagrangian Markov operators $M_L = [m_{r,s}]$ governs the rank data $\Omega^S$ on $\mathcal{D}^*$ in that 
\begin{equation}
m^{(i)}_{\mathrm{rk}(\omega),s} = \frac{\sum_{q \in \mathcal{P}_i^{k,*}(X)-\delta}\left| \left\{ \chi \in \eta^{-1}_{\delta,q}(\omega) : \mathrm{rk}(\chi)=s\right\}\right|}{\sum_{q \in \mathcal{P}_i^{k,*}(X)-\delta}\left| \left\{ \eta^{-1}_{\delta,q}(\omega)\right\}\right|} 
\end{equation}
where $M^i= [m_{r,s}^{(i)}]$.

Moreover, every function $\mathcal{L}$ satisfying Theorem 8.1 of \cite{KMR14} is a convergence rate for $(\Omega^S, M_L)$.
\end{proposition}

The stationary distribution of the alternating mod $p$ Lagrangian Markov operator $M_L$ can be computed from applying Proposition 2.4 of \cite{KMR14}.

\begin{proposition} \label{prop:equidist-ML}
    Let $f: \mathbb{Z}_{\geq 0} \to [0,1]$ be a map which lies in the set 
    $$\left\{
\text{maps}~ f: \mathbb{Z}_{\geq 0} \to \mathbb{R} ~\Big|~ ||f||:= \sum_{s \in \mathbb{Z}_{\geq 0}} |f(s)| ~ \text{converges}
\right\}.$$
    Denote by $\rho(f)$ the quantity
    $$\rho(f) := \sum_{s \in \mathbb{Z}_{\geq 0, \text{even}}} |f(s)|.$$ Then
    \begin{equation}
        \lim_{k \to \infty} M_L^k(f) = \rho(f) \cdot \mathbf{E}_{even} + (1-\rho(f)) \cdot \mathbf{E}_{odd},
    \end{equation}
    where
    \begin{align*}
        \mathbf{E}_{even}(s) &:= \begin{cases}
            \prod_{k=0}^\infty \frac{1}{1 + 3^{-k}} \cdot \prod_{k=1}^{\frac{s}{2}} \frac{3}{3^k - 1} &\text{ if } s \equiv 0 \mod 2 \\
            0 &\text{ otherwise}
        \end{cases} \\
        \mathbf{E}_{odd}(s) &:= \begin{cases}
            \prod_{k=0}^\infty \frac{1}{1 + 3^{-k}} \cdot \prod_{k=1}^{\frac{s-1}{2}} \frac{3}{3^k - 1} &\text{ if } s \equiv 1 \mod 2 \\
            0 &\text{ otherwise}
        \end{cases}
    \end{align*}
\end{proposition}

\begin{remark}
The infinite product above can be estimated as follows:
    $$\prod_{k=0}^\infty \frac{1}{1 + 3^{-k}} \approx 0.3195022.$$
\end{remark}

For small values of $s$, $\mathbf{E}_{even}(s)$ and $\mathbf{E}_{odd}(s)$, are approximately as follows.
\smallskip

\begin{center}
    \begin{tabular}{c||cc}
        $s$ & $\mathbf{E}_{even}(s)$ & $\mathbf{E}_{odd}(s)$ \\
        \hline \hline 
        0 & 0.319502 & 0 \\
        1 & 0 & 0.319502 \\
        2 &0.479253 & 0\\
        3 & 0& 0.479253 \\
        4 & 0.17972 & 0\\
        5 & 0 & 0.17972 \\
        6 &0.020737 & 0\\
        7 & 0 & 0.020737  \\
        8 &0.00078 & 0 \\
        9 & 0 & 0.00078\\
        10 & $0.94 \times 10^{-5} $&  0\\
    \end{tabular}
\end{center}

\section{Fan Structures Following \texorpdfstring{\cite{KMR14}}{[KMR]}}\label{sec:fans}

We begin by recalling some definitions from \cite{KMR14} which will be used this section and the following section where we prove our main results. As in the previous section, many of these definitions are stated for a general set $\mathcal{D}^*$ consisting of products of primes. In applications we'll then make a choice, setting $\mathcal{D}^*$  to be $\mathcal{D}^A, \mathcal{D}^B,$ or $\mathcal{D}^C$ corresponding to products of primes respecting cases $A,B,C$ of our parametrization for $S_3$-cubic extensions by quadratic resolvent. These dependencies on cases $A,B,C$ are important, but rather mild, so we delay the case-by-case analysis for as long as possible. 

\begin{definition}[As in Definition 10.1 of \cite{KMR14}]
    For any $\chi \in C(k) = \Hom(G_k, \mu_3)$ and $v$ a place of $k$, write $\chi_v$ for the restriction of $\chi$ to $G_{k_v}$. 
    \begin{itemize}
    \item For $\mathfrak{d} \in \mathcal{D}^*$, define
    \begin{equation}
        C(\mathfrak{d}) := \{\chi \in C(k) : \chi \text{ is unramified at all }~\mathfrak{q} \mid \mathfrak{d}~\text{and unramified outside } \Sigma(\mathfrak{d}) \cup \mathcal{P}_0^{k,*}\}.
    \end{equation}
    \item For $\mathfrak{d} \in \mathcal{D}^*$, we denote by $\eta_\mathfrak{d}: C(\mathfrak{d}) \to \Omega_\mathfrak{d}$ the product of restriction maps. In other words, we have $\eta_\mathfrak{d}(\chi) = (\chi_v)_{v \in \Sigma(\mathfrak{d})}$.
        \item For $X  > 0$, 
        \begin{equation}
            C(X) := \{\chi \in C(k) : \chi \text{ is unramified outside of } \Sigma \cup \{\mathfrak{q} : \mathbf{N}\mathfrak{q} \leq X\}\}
        \end{equation}
        \item Finally, define the intersection of the above two sets as
        \begin{equation}
            C(\mathfrak{d}, X) := C(\mathfrak{d})  \cap C(X).
        \end{equation}
    \end{itemize}
\end{definition}

\begin{proposition} [As in Proposition 10.7 of \cite{KMR14}]
    Suppose $\mathfrak{d} \in \mathcal{D}^*$, $\mathcal{L}$ is a function satisfying Theorem 8.1 of \cite{KMR14}, and $X > \mathcal{L}(\mathbf{N} \mathfrak{d})$. Then $\eta_\mathfrak{d}: C(\mathfrak{d},X) \to \Omega_\mathfrak{d}$ is surjective. Furthermore, for every $\omega \in \Omega_\mathfrak{d}$,
    \begin{equation}
        \frac{\#\{\chi \in C(\mathfrak{d},X) \; | \; \eta_\mathfrak{d}(\chi) = \omega\}}{\# \Omega_\mathfrak{d}} = \frac{1}{\# \Omega_\mathfrak{d}}.
    \end{equation}
\end{proposition}

\begin{definition}[As in Definition 3.14 of \cite{KMR14}]
Let $\mathcal{L}: [1,\infty) \to [1,\infty)$ be a non-decreasing function.
Let $\{L_n(Y)\}_{n \geq 1}$ be a sequence of real valued functions defined as follows:
\begin{align*}
    L_1(Y) &:= \mathcal{L}(Y) \\
    L_{n+1}(Y) &:= \max \left\{ \mathcal{L} (\prod_{j \leq n} L_j(Y)), YL_n(Y)\right\}, \; n \geq 1.
\end{align*}
Given $m, w \in \mathbb{Z}$ and a choice of a set $\mathcal{D}^*$ of square-free products of places $v \in \mathcal{P}_1^{k,*} \cup \mathcal{P}_2^{k,*}$, define the \textit{fan} as follows
\begin{equation}
    \mathcal{D}^*_{m,w,X} := \left\{ \mathfrak{d} \in \mathcal{D}^* \; | \; \mathfrak{d} = \{q_j\}_{j=1}^m \text{ with } \mathbf{N}(q_j) < L_j(X) \; \forall j, \text{ and } \sum_{j=1}^m \dim_{\mathbb{F}_3}E[3](k_{q_j}) = w \right \}.
\end{equation}
\end{definition}

\begin{definition}[As in Definition 11.4 of \cite{KMR14}]\label{def:fan-structure}
Let $\{L_n(Y)\}_{n \geq 1}$ be a sequence of real valued functions defined as in Definition 3.14 of \cite{KMR14}. For $m,w \geq 0$, and a choice of a set of square-free products of places $\mathcal{D}^*$, we define the \emph{fan structure} on $\mathcal{C}(k)$ as 
\begin{equation}
\mathcal{B}_{m,w,\mathcal{D}^*,X} := \bigsqcup_{\mathfrak{d} \in \mathcal{D}^*_{m,w,X}} \mathcal{C}(\mathfrak{d}, \mathcal{L}(L_{m+1}(X))) \subset \mathcal{C}(k).
\end{equation}
\end{definition}
As usual we use $\mathcal{D}^*$ for a fixed choice of either $\mathcal{D}^A, \mathcal{D}^B,$ or $\mathcal{D}^C$. 
\begin{remark}\label{rmk:fan}
    The ordering imposed on the fan structure is non-standard in that the ordering on $\mathfrak{d}$ is not by norm, and the final ordering on $S_3$-cubic extensions in our theorems is not by, say, discriminant. 
    
    Under the natural ordering with $\mathbf{N}(\mathfrak{d}) \leq X$, the number of possible primes in the support of $\mathfrak{d}$ grows like $\log \log X$. Under the ordering by $\mathcal{L}$, we allow sufficiently many primes in the support of $\mathfrak{d}$ so that the number of such possible primes grow \textit{linearly} with $m$, and the upper bound on the norms of such primes are functions of $X$.
    
    The faster growth in the number of admissible $\mathfrak{d}$ allows one to actually obtain the convergence of a Markov process. Ordering $\mathfrak{d}$ by norm in this regime is \textit{not} sufficient to obtain the desired convergence.

    For example, suppose the function $\mathcal{L}$ is given by $\mathcal{L}(X) = \log X$. Then the fan $\mathcal{D}_{m,\omega,X}^*$ is a set of composite numbers $\mathfrak{d} \in \mathcal{D}^*$ whose set of distinct prime factors is $\{q_1, q_2, \cdots, q_m\}$, and each factor $q_i$ satisfies $\mathbf{N}(\mathfrak{d}) \leq X^{i-1} \log X$.
\end{remark}

We are now ready to state the theorem from which our result all follow. The theorem should be thought of as a version of Theorem 11.6 in \cite{KMR14} for which all the choice of $\mathcal{D}^*$ (in the notation of \cite{KMR14}) and all dependent structures have been further restricted to respective cases $A,B,C$ of the parameterization in Section \ref{section:S3-param}.

\begin{definition}[As in Definition 3.9 of \cite{KMR14}]\label{def:Efrakd}
    Given a square-free product of places $\mathfrak{d}$ of $k$, we denote by $E_{\mathfrak{d}}:\mathbb{Z}_{\geq 0} \to \mathbb{R}$ the probability distribution defined as
    \begin{equation}
        E_{\mathfrak{d}}(s) := \frac{\#\{\omega \in \Omega_{\mathfrak{d}} \; | \; \dim_{\mathbb{F}_3}\Sel(E[3],\omega)_F = s\}}{\# \Omega_\mathfrak{d}}.
    \end{equation}
\end{definition}

\begin{theorem} \label{thm:our-11.6}
    If $m,w,s\geq 0$ and $\bigcup_X \mathcal{D}^*_{m,w,X}$ is non-empty, then
    \begin{equation}\label{eq:ourKMR116}
    \lim_{X \to \infty} \frac{
        |\{\chi \in \mathcal{B}_{m,w,\mathcal{D}^*,X} : \dim_{\mathbb{F}_3}\Sel(E[3],\eta_\mathfrak{d}(\chi))_F = s\}|    
    }{
        |\mathcal{B}_{m,w,\mathcal{D}^*,X}|
    } = (M_L^w (E_1))(s).
    \end{equation}
\end{theorem}

\begin{proof}
    By Proposition 10.7 of \cite{KMR14}, for any squarefree $\mathfrak{d} \in \mathcal{D}^*$ and $X > \mathcal{L}(N \mathfrak{d})$, the restriction of characters in $C(\mathfrak{d},X)$ to $\Omega_1$ is equidistributed. We note that such equidistribution results are obtained regardless of the choice of the virtual $3$-Selmer units $u \in S_3(F_z)[T]$, because the set $\Sigma$ already includes the places dividing any of the elements in $S_3(F_z)[T]$. 
    The proof of the theorem is the same as Theorem 11.6 in \cite{KMR14} with the appropriate $\mathcal{D}^*$ in place of $\mathcal{D}$.
\end{proof}

\section{Proof of the main theorem}\label{sec:proof-of-main-thms}

We first introduce a couple of notations that will be used in the statement of the main theorem.

\begin{definition}\label{def:rhoE}
    We denote by $\rho_E \in [0,1]$ the probability that $E_1$ is even, i.e.
    \begin{equation}
        \rho_E := \sum_{s \in \mathbb{Z}_{\geq 0, \text{even}}} E_1(s).
    \end{equation}
\end{definition}
Above, $E_1$ is as in Definition \ref{def:Efrakd} with $\mathfrak{d}$ taken to be the empty product of primes in $k$.
\begin{definition} \label{def:FAN}
Let $\mathcal{D}^*$ be one of $\mathcal{D}^A, \mathcal{D}^B, \mathcal{D}^C$ from Definition \ref{def:DA_DB_DC}.
For every $m \geq 0$ and $X > 0$, let 
\begin{equation}
    \mathcal{B}_{m,*}(X) := \bigcup_{w} \mathcal{B}_{m,w,\mathcal{D}^*,X}
\end{equation}
be a collection of $S_3$-cubic extensions $K$ of $k$ with a fixed quadratic resolvent field $F/k$. The right hand side is a union of fan structures as in Definition \ref{def:fan-structure}.
\end{definition}

We now have all the ingredients to state and prove the main theorem. Theorems \ref{theorem:mainA} and \ref{thm:mainB} are immediate consequences of the following result.

\begin{theorem}\label{thm:main-full}
Fix an admissible quadratic extension $F/k$.
    Let $E$ be an elliptic curve over a number field $k$ for which $\text{Gal}(k(E[3])/k) \supset \text{SL}_2(\mathbb{F}_3)$,
    and $F$ and $k(E[3])$ are linearly disjoint over $k$.
    Then for every $s \geq 0$, we have
    \begin{equation}
        \lim_{m \to \infty} \lim_{X \to \infty} \frac{\#\{ K \in \mathcal{B}_{m,*}(X) \; | \; \dim_{\mathbb{F}_3} \Sel_{1-\sigma_K}(B_{K/k}/k) = s\}}{\# \mathcal{B}_{m,*}(X)} = \rho_E \cdot \mathbf{E}_{even}(s) + (1 - \rho_E) \cdot \mathbf{E}_{odd}(s).
    \end{equation}
\end{theorem}
\begin{proof}
    Given an $S_3$-cubic extension $K/k$, denote by $\chi_K \in \text{Hom}(G_k,\mu_3)$ the global character whose fixed field is $K$. Using the parametrization of $S_3$-cubics in Section \ref{section:S3-param}, we can find a squarefree product of places $\mathfrak{d} \in \mathcal{D}^*$ such that $\chi_K \in \mathcal{C}(\mathfrak{d})$. 

    Denote by $\eta_\mathfrak{d}: C(\mathfrak{d}) \to \Omega_\mathfrak{d}$ the product of restriction maps of global characters. By Proposition \ref{prop:random-matrix-model} and Shapiro's Lemma,
    \begin{equation}
        \Sel_{1-\sigma_K}(B_{K/k}/k) \cong \Sel(E[3],\eta_\mathfrak{d}(\chi_K))_F.
    \end{equation}
    Hence, the problem of computing the probability distribution of dimensions of $\Sel_{1-\sigma_K}(B_{K/k}/k)$ can be reduced to that of computing the probability distribution of dimensions of local Selmer groups $\Sel(E[3],\eta_\mathfrak{d}(\chi_K))_F$. The theorem then follows from applying Proposition \ref{prop:equidist-ML} and Theorem \ref{thm:our-11.6}. The parity $\rho_E$ appears from the fact that the Markov operator $M_L$ preserves the parity of the dimensions of local Selmer groups $\Sel(E[3],\eta_\mathfrak{d}(\chi_K))_F$.
\end{proof}

As a corollary, we are able to obtain quadratic exponential decay in the probability that an elliptic curve achieves arbitrarily large rank growths over fan structures of $S_3$-cubic extensions.
\begin{corollary}\label{cor:exponential-decay}
Fix an admissible quadratic extension $F/k$.
    Let $E$ be an elliptic curve over a number field $k$ for which $\text{Gal}(k(E[3])/k) \supset \text{SL}_2(\mathbb{F}_3)$,
    and $F$ and $k(E[3])$ are linearly disjoint over $k$.

    Let $C := \prod_{k=1}^\infty \frac{1}{1-3^{-k}} \approx 1.785312342$. Then for every $s \geq 4$, we have
    \begin{equation}
    \lim_{m \to \infty} \lim_{X \to \infty} \frac{\#\{K \in \mathcal{B}_{m,*}(X) \; | \; \text{rk}(E(K)) - \text{rk}(E(k)) \geq s \}}{\# \mathcal{B}_{m,*}(X)} < \begin{cases} 
    C \cdot 3^{-\frac{s(s-2)}{8}} &\text{ if } s \text{ even } \\
    C \cdot 3^{-\frac{(s-1)(s-3)}{8}} &\text{ if } s \text{ odd}.
    \end{cases}
    \end{equation}
\end{corollary}
\begin{proof}
    Without loss of generality, assume $s$ is even and $s \geq 4$. We have
    \begin{equation*}
        \frac{\#\{K \in \mathcal{B}_{m,*}(X) \; | \; \text{rk}(E(K)) - \text{rk}(E(k)) \geq s \}}{\# \mathcal{B}_{m,*}(X)} \leq \frac{\#\{K \in \mathcal{B}_{m,*}(X) \; | \; \dim_{\mathbb{F}_3}\Sel_{1-\sigma_K}(B_{K/k}/k) \geq s \}}{\# \mathcal{B}_{m,*}(X)}
    \end{equation*}
    Taking arbitrarily large values of $m$ and $X$ on both sides of the equation, we obtain
    \begin{align}
        \begin{split}
        & \; \; \; \; \lim_{m \to \infty} \lim_{X \to \infty }\frac{\#\{K \in \mathcal{B}_{m,*}(X) \; | \; \text{rk}(E(K)) - \text{rk}(E(k)) \geq s \}}{\# \mathcal{B}_{m,*}(X)} \\
        &\leq \prod_{k=0}^\infty \frac{1}{1 + 3^{-k}} \cdot \left( \sum_{j = \frac{s}{2}}^\infty \prod_{k=1}^{j} \frac{3}{3^k - 1} \right) \\
        &\leq \frac{3^{\frac{s}{2}}}{3^{\frac{s}{2}} - 1} \cdot \prod_{k=1}^{\frac{s}{2}-1} \frac{1}{3^k - 1} \cdot \prod_{k=0}^\infty \frac{1}{1 + 3^{-k}} \cdot \left( 1 + \sum_{j=\frac{s}{2}+1}^\infty \prod_{k = \frac{s}{2} + 1}^j \frac{3}{3^k - 1} \right) \\
        &\leq \left(\frac{3^{\frac{s}{2}}}{3^{\frac{s}{2}}-1} \cdot \prod_{k=1}^{\frac{s}{2}-1} \frac{1}{1 - 3^{-k}} \right) \cdot 3^{-\frac{s(s-2)}{8}} \\
        &\leq \prod_{k=1}^{\infty} \frac{1}{1 - 3^{-k}} \cdot 3^{-\frac{s(s-2)}{8}}.
        \end{split}
    \end{align}
    The upper bound for the case when $s$ is odd can be obtained from using the upper bound for the probability that $\text{rk}(E(K)) - \text{rk}(E(k)) \geq s-1$.
\end{proof}

\nocite{*}
\bibliographystyle{alpha}
\bibliography{main}

\end{document}